\documentclass[preprint,12pt]{elsarticle}

\usepackage{amssymb}
\usepackage{amsthm}
\usepackage{amsmath}
\usepackage{amsfonts}
\usepackage{multirow}
\usepackage{mathrsfs}
\usepackage{xcolor}%
\usepackage{graphicx}
\usepackage{float}
\usepackage{subfig}
\usepackage{verbatim}

\numberwithin{equation}{section}
\numberwithin{figure}{section}
\newtheorem{theorem}{Theorem}[section]
\newtheorem{remark}{Remark}[section]%
\newtheorem{lemma}{Lemma}[section]
\newtheorem{corollary}{Scheme}[section]
\journal{Journal of Computational Physics}

\begin{document}

\begin{frontmatter}
\title{ Improving the accuracy and consistency of the 
energy quadratization method with an energy-optimized technique }

\author[Label1]{Xiaoqing Meng}
    \ead{202320305@mail.sdu.edu.cn}
\author[Label1]{Aijie Cheng\corref{cor1}}
    \ead{aijie@sdu.edu.cn}
    \cortext[cor1]{Corresponding author.}
\author[Label2]{Zhengguang Liu\corref{cor1}}
    \ead{liuzhg@sdnu.edu.cn}
\affiliation[Label1]{
            organization={School of Mathematics},%Department and Organization
            addressline={Shandong University}, 
            city={Jinan},
            postcode={250100}, 
            state={Shandong},
            country={China}}
\affiliation[Label2]{
            organization={School of Mathematics and Statistics},%Department and Organization
            addressline={Shandong Normal University}, 
            city={Jinan},
            postcode={250358}, 
            state={Shandong},
            country={China}}
\begin{abstract}
    We propose an energy-optimized invariant energy quadratization method to solve the gradient flow models in this paper, 
    which requires only one linear energy-optimized step to correct the auxiliary variables on each time step. 
    In addition to inheriting the benefits of the baseline and relaxed invariant energy quadratization method,
    our approach has several other advantages. %, including:
    Firstly, in the process of correcting auxiliary variables, we can directly solve linear programming problem by 
    the energy-optimized technique, which greatly simplifies the nonlinear optimization problem 
    in the previous relaxed invariant energy quadratization method.
    Secondly, we construct new linear unconditionally energy stable schemes 
    by applying backward Euler formulas and Crank-Nicolson formula,
    so that the accuracy in time can reach the first- and second-order. %to discrete time
    Thirdly, comparing with relaxation technique, the modified energy obtained 
    by energy-optimized technique is closer to the original energy,
    meanwhile the accuracy and consistency of the numerical solutions can be improved.
    Ample numerical examples have been presented to demonstrate the accuracy, efficiency and 
    energy stability of the proposed schemes.
\end{abstract}

\begin{keyword}
    Invariant energy quadratization \sep Gradient flow models \sep  energy-optimized technique 
    \sep  Unconditionally energy stable  \sep  Numerical experiments
    \MSC[2020] 65M12 \sep 35K20 \sep  35K35 \sep  35K55 \sep  65Z05
%% keywords here, in the form: keyword \sep keyword
%% PACS codes here, in the form: \PACS code \sep code
%% MSC codes here, in the form: \MSC code \sep code
%% or \MSC[2008] code \sep code (2000 is the default)
\end{keyword}

\end{frontmatter}

%% \linenumbers
\thispagestyle{empty}

%% main text
\section{ Introduction }
\biboptions{numbers,sort&compress} % 
Many physical problems such as interfacial dynamics \cite{Anderson1997,HOU2019307}, 
crystal growth \cite{KOBAYASHI1993410,PhysRevLett}, image inpainting \cite{Bai2023AGM,ZHANG202374}, 
thin films \cite{Giacomelli2001VariatonalFF,Cheng2010}, 
and polymers \cite{Fraaije1993}
can be modeled by gradient flow models which comply with the second law of thermodynamics.
As a special kind of nonlinear partial differential equations, the gradient flow model can be derived from
%obtained according to 
the mobility and the variational derivative of free energy when the total free energy is known.
Generally, consider the spatial-temporal domain $\Omega_t:=\Omega \times \left(0,T\right]$, 
$\Omega \subset \mathbb{R}^d, d=1,2,3$
and the state variable $\phi =\phi \left(\boldsymbol{x} , t\right) :\Omega _t\mapsto \mathbb{R}$,
the total free energy $E(\phi)$ consists of a quadratic term $E_{\mathcal{L}}(\phi)$ and a strongly nonlinear term $E_{\mathcal{N}}(\phi)$,
namely
\begin{align}
    \label{E0}
    E\left( \phi \right) =E_{\mathcal{L}}\left( \phi \right) +E_{\mathcal{N}}\left( \phi \right) 
    =\frac{1}{2}\int_{\Omega}{\phi \mathcal{L}\phi}\, \text{d}\boldsymbol{x}+\int_{\Omega}{F\left( \phi \right)}\,\text{d}\boldsymbol{x},
    %E(\phi)=E_{\mathcal{L}}(\phi)+E_{\mathcal{N}}(\phi)=(\mathcal{L}\phi,\phi)+(F(\phi),1),
\end{align}
where $F(\phi)$ is a nonlinear potential function, and $\mathcal{L}$ is a linear self-adjoint non-negative definite operator.
Then the corresponding general gradient flow system of \eqref{E0} reads
\begin{equation}
    \label{PF}
    \begin{aligned}
    {}&\frac{\partial \phi}{\partial t}=-M\mathcal{G}\mu ,\\
    {}&\mu =\mathcal{L}\phi +f\left( \phi \right) ,
    \end{aligned}
\end{equation}
with the initial condition $\phi(\boldsymbol{x},0)$ and thermodynamically consistent 
boundary conditions (such as periodic and homogeneous Neumann boundaries),
%periodic or homogeneous Neumann boundary conditions,
where $f(\phi)=F'(\phi)$, $\mu =\frac{\delta E}{\delta \phi}$ is the chemical potential, $M$ is a mobility constant
and $\mathcal{G}$ is a non-negative definite operator reflecting the dissipation path. 
The gradient flow system $\eqref{PF}$ preserves the energy dissipation law %the important property of, 
%i.e.,
\begin{align}
	\label{dE}
	\frac{dE\left( \phi \right)}{dt}=\left( \frac{\delta E}{\delta \phi},\frac{\partial \phi}{\partial t} \right) =\left( \mathcal{L}\phi +f\left( \phi \right) ,-M\mathcal{G}\mu \right) =-\left( \mu ,M\mathcal{G}\mu \right) \le 0,
\end{align}
where $\left( \cdot ,\cdot \right)$ denotes the $L^2$ product, i.e. 
$\left( f,g \right) =\int_{\Omega}{fg}\,\text{d}\boldsymbol{x}, \forall f,g\in L^2\left( \Omega \right) $.

In order to eliminate non-physical numerical solutions, the energy dissipation law must be preserved 
at the discrete level when designing numerical schemes for nonlinear gradient flow models. %dissipative
In recent years, there has been tremendous interest in designing 
efficient and accurate energy stable schemes for the nonlinear dissipation systems.
To this end, many efficient schemes have been developed,
including convex splitting method \cite{Elliott1993,Baskaran2013}, %is to decompose the energy functional into 
%the sum of convex and concave parts, the term corresponding to the convex part of the equation 
%is treated implicitly and the concave part is treated explicitly;
the exponential time differencing (ETD) \cite{Hochbruck2005,Hochbruck2010ExponentialI} method, %is a time-integral approximation scheme based on exponential integral factors;
stabilized semi-implicit (SSI) approach \cite{Feng2013StabilizedCS,Shen2010NumericalAO}, %adds a stability term of the same order 
%as the truncation error of the original linear semi-implicit scheme to ensure the energy stability;
and the new Lagrange multiplier method \cite{Cheng2019ANL}.
%is designed to solve two linear PDEs with constant coefficients and one nonlinear algebraic equation.
Recently, a numerical approach named invariant energy quadratization (IEQ) or energy quadratization (EQ) method is proposed 
to solve thermodynamically consistent PDE problems,
and by utilizing auxiliary variables, the original PDE problems are reformulated into equivalent
PDE problems.
Based on the baseline IEQ method, linear, second-order even any high-order, easy-to-implement,
and unconditionally energy stable schemes are constructed 
for complex gradient flow models 
\cite{YANG2016294,YANG201880,Yang2017,Zhao2017NumericalAF,Zhao2017ANL,Zhao2018,Yang2020ConvergenceAF,GONG2019,Gong2020}.
Subsequently, the scalar auxiliary variable (SAV)-type approaches \cite{Shen2019,Liu2020,Huang2020,JIANG2022110954} are proposed
and can also be applied to solve a class of gradient flow models.

However, for the baseline IEQ method, the truncation errors are introduced 
in the long-time numerical simulation process, so that the numerical solutions of 
auxiliary variables are no longer equivalent to their original continuous definitions.
%the discrete values of the auxiliary variables are not directly related to the definition of the continuous level.
We know that the baseline IEQ method preserves the modified energy dissipation law 
with respect to the auxiliary variables instead of the original variables,
%a modified energy,
%based on the auxiliary variables, 
which means the original energy corresponding to the numerical solutions is not necessarily monotonically decreasing. 
This inconsistency 
%is not only introduces calculation error in the long-time numerical simulation process,and 
also leads to a loss of accuracy when the time step is not small enough.
In order to overcome the adverse effects of this inconsistency, the pioneering relaxation techniques 
\cite{ZHAO2021107331,JIANG2022110954,Zhang2023EfficientAA,Huang2024ACO} are used to correct
the numerical solutions of the auxiliary variables, so that the modified energy is closer to the original energy.
In this paper, inspired by a novel and essential energy-optimal technique based on the SAV method, i.e. EOP-SAV method \cite{liu2023novel}, 
we propose an energy-optimized approach based on the baseline IEQ method, and construct 
first- and second-order (in time) unconditionally energy stable schemes which can solve 
a large class of gradient flow problems.
We effectively penalize the inconsistency between the numerical solution of auxiliary variables 
and their continuous definitions by using the proposed energy-optimized IEQ (EOP-IEQ) method
and the main advantages of our approach are listed as follows:
\begin{itemize}
    \item  The EOP-IEQ method inherits all the advantages of the baseline IEQ method \cite{YANG2016294}
     and relaxed IEQ (REQ) method \cite{ZHAO2021107331}; %\cite{YANG2016294}\cite{ZHAO2021107331}
    \item Our energy-optimized technique makes the calculation process of updating auxiliary variables more simpler and efficient;
    \item Compared with the IEQ and REQ methods, the modified energy 
    obtained by the energy-optimized technique is closer to the original energy.
\end{itemize}
All numerical schemes based on the EOP-IEQ method are linear unconditionally energy stable and easy-to-implement. 
In numerical examples of various gradient flow models, the modified energy corresponding of the numerical solutions 
based on the auxiliary variables 
almost satisfy the original energy dissipation law.

The rest of this paper is organized as follows. 
In Section 2, we briefly review the baseline IEQ and REQ methods for gradient flow models. 
In Section 3, we propose the EOP-IEQ method for general dissipative systems,
and prove that the new first- and second-order (in time) numerical schemes are unconditionally energy stable.
In Section 4, %we apply the EOP-IEQ method to a variety of gradient flow models.
we verify its optimality by comparing the EOP-IEQ method with the baseline 
IEQ and REQ methods, and provide a large number of numerical experiments 
to verify  the accuracy, efficiency and energy stability of
the proposed schemes.

\section{ A brief review of the IEQ-type methods }
In this section, we review the baseline IEQ method which considered by
Yang in \cite{YANG2016294} and the relaxed IEQ (REQ) approach proposed in \cite{ZHAO2021107331} by Zhao.

\subsection{The baseline IEQ method}
Firstly, we need to assume that the bulk free energy $F(\phi)$
is bounded from below which means $F(\phi) > - C_*$ for a positive constant $C_*$. 
Introduce an auxiliary variable 
\begin{align}
	\label{eq_q}
	q\left( \boldsymbol{x},t \right) := Q\left( \phi\left( \boldsymbol{x},t \right) \right) =\sqrt{F\left( \phi\left( \boldsymbol{x},t \right) \right) +C_0},\quad C_0>C_*,
\end{align}
where the positive constant $C_0$ makes sure $q(\boldsymbol{x},t)$ a well defined real function.
Then we rewrite the gradient flow system \eqref{PF} as the following equivalent system
\begin{equation}
	\label{IEQ}
	\begin{aligned}
		{}&\frac{\partial \phi}{\partial t}=-M\mathcal{G}\mu ,\\
		{}&\mu =\mathcal{L}\phi +\frac{q}{Q(\phi)}f\left( \phi \right) ,\\
		{}&\frac{\partial q}{\partial t}=\frac{f\left( \phi \right)}{2Q(\phi)}\frac{\partial \phi}{\partial t}.
	\end{aligned}
\end{equation}
Using the IEQ method above, we convert the free energy \eqref{E0} into the following quadratic form 
%by introducing the auxiliary variable $q(x,t)$,
\begin{align}
	\label{E_{1}^{n+1}}
	\bar{E}( \phi ,q ) =\frac{1}{2}( \mathcal{L}\phi ,\phi ) +(q,q)-C_0|\Omega|,
\end{align}
It is not difficult to obtain the following modified energy dissipation law
\begin{align}
	\label{dE_{1}^{n+1}}
	\frac{d\bar{E}\left( \phi,q \right)}{dt}=-M\left( \mathcal{G}\mu ,\mu \right) \le 0.
\end{align}
\begin{remark}
With the IEQ method, numerical algorithms can be introduced to solve the equivalent model \eqref{IEQ}
instead of solving the original model \eqref{PF}. %, since \eqref{PF} and \eqref{IEQ} are equivalent.
%By introducing the IEQ method, a numerical algorithm can be used to solve the equivalent model \eqref{IEQ}, 
%so that an approximate solution to the original model \eqref{PF} can be obtained.
\end{remark}
Now we introduce some notations that will be used throughout the paper. Let 
$L^p(\Omega)$ denote the usual Lebesgue space on $\Omega$ with the norm $\|\cdot\|_{L^p}$.
$W^{k,p}(\Omega)$ stands for the standard Sobolev spaces equipped with the standard Sobolev norms 
$\|\cdot\|_{k,p}$. For $p = 2$, we write $H^k(\Omega)$ for $W^{k,2}(\Omega)$
and the corresponding norm is $\|\cdot\|_k$, where $\|\cdot\|_0$ is simply represented by $\|\cdot\|$. \par
We discretize time domain $[0,T]$ into uniform meshes
%Before giving a semi-discrete formulation, we discretize the temporal domain $[0,T]$ into equally distanced meshes
$0=t_0<t_1<\cdots<t_N=T$ with $t_n=n\tau$, $\tau=T/N$,
and use $(\bullet )^{n+1}$ to denote the numerical approximation of $(\bullet )$ at $t_{n+1}$. %\bullet
With above notations, we obtain the following semi-discrete $k$th-order IEQ/BDF$k$ schemes 
based on the backward Euler formulas.
\begin{corollary}[IEQ/BDF$k$ schemes \cite{YANG2016294}] 
    For initial datas, we set $\phi ^0=\phi \left( \boldsymbol{x},0 \right)$ and $q^0=\sqrt{F\left( \phi ^0 \right) +C_0}$,
    we can update $(\phi^{n+1}, q^{n+1})$ via the following step:
\begin{equation}
	\label{IEQ1}
	\begin{aligned}
		{}&\frac{\alpha _k\phi ^{n+1}-A_k\left( \phi ^n \right)}{\tau}=-M\mathcal{G}\mu ^{n+1},\\
		{}&\mu ^{n+1}=\mathcal{L}\phi ^{n+1}+\frac{q^{n+1}}{Q\left( B_k\left( \phi ^n \right) \right)}f\left( B_k\left( \phi ^n \right) \right) ,\\
		{}&\frac{\alpha _kq^{n+1}-A_k\left( q^n \right)}{\tau}=\frac{f\left( B_k\left( \phi ^n \right) \right)}{2Q\left( B_k\left( \phi ^n \right) \right)}\frac{\alpha _k\phi ^{n+1}-A_k\left( \phi ^n \right)}{\tau},
	\end{aligned}
\end{equation}
%Here $\alpha_k$, $A_k$ and $B_k$ are different for $k$th-order schemes. For example, the different values for $k=1,2$ and the corresponding modified energies are as follows:
where $Q\left( B_k\left( \phi ^n \right) \right) =\sqrt{F\left( B_k\left( \phi ^n \right) \right) +C_0}$.
The values of $\alpha_k$, $A_k(\phi^n)$ and $B_k(\phi^n)$ of $k=1,2$ are as follows:
\begin{itemize}
    \item[] \rm{First-order}:
    \begin{align*}
    \alpha _1=1,\quad A_1\left( \phi ^n \right) =\phi ^n,\quad B_1\left( \phi ^n \right) =\phi ^n.
    %E^n{}&=\frac{1}{2}\left( \mathcal{L}\phi ^n,\phi ^n \right) +\lVert q^n \rVert ^2-\left( B,1 \right), 
    \end{align*}
    \item[] \rm{Second-order}:
    \begin{align*}
        \alpha _2={}&\frac{3}{2},\quad A_2\left( \phi ^n \right) =2\phi ^n-\frac{1}{2}\phi ^{n-1},\quad B_2\left( \phi ^n \right) =2\phi ^n-\phi ^{n-1}.
    %    E^n={}&\frac{1}{4}\left( ( \mathcal{L}\phi^n ,\phi^n)+( \mathcal{L}( 2\phi^n -\phi^{n-1} ) ,2\phi^n -\phi^{n-1} ) \right) \\
    %    {}&+\frac{1}{2}\left( \|q^n\|^2+\|2q^n-q^{n-1}\|^2\right) -(B,1).
    \end{align*}
\end{itemize}
\end{corollary}
\begin{lemma}[\cite{YANG2016294}]
    The IEQ/BDF$k$ $(k=1,2)$ schemes \eqref{IEQ1} are unconditionally energy stable in the sense that 
    \begin{itemize}
        \item[] \rm{ First-order} :
        \begin{comment}
        $$
        \bar{E}\left( \phi ^{n+1},q^{n+1} \right) -\bar{E}\left( \phi ^n,q^n \right) \le -\tau \left( \mathcal{G}\mu ^{n+1},\mu ^{n+1} \right) ,
        $$
        where 
        $
        \bar{E}\left( \phi ^{n+1},q^{n+1} \right) =\frac{1}{2}\left( \mathcal{L}\phi ^{n+1},\phi ^{n+1} \right) +\lVert q^n \rVert ^2-\left( B,1 \right). 
        $
        \end{comment}
        %\begin{comment}
        $$
        \frac{1}{2}\left( \mathcal{L}\phi ^{n+1},\phi ^{n+1} \right) +\lVert q^{n+1} \rVert ^2-\frac{1}{2}\left( \mathcal{L}\phi ^n,\phi ^n \right)
         -\lVert q^n \rVert ^2\le -\tau \left( \mathcal{G}\mu ^{n+1},\mu ^{n+1} \right)\le 0 .
        $$
        %\end{comment}
        \item[] \rm{ Second-order} :
        \begin{comment}
        $$
        \bar{E}\left( \phi ^{n+1},\phi ^n,q^{n+1} \right) -\bar{E}\left( \phi ^n,\phi ^{n-1},q^n,q^{n-1} \right) \le -\tau \left( \mathcal{G}\mu ^{n+1},\mu ^{n+1} \right) ,
        $$
        where
        $
        \bar{E}\left( \phi ^{n+1},\phi ^n,q^{n+1},q^{n-1} \right) =\frac{1}{4}\left( \mathcal{L}\phi ^{n+1},\phi ^{n+1} \right) +\frac{1}{2}\lVert q^{n+1} \rVert ^2+\frac{1}{4}\left( \mathcal{L}\left( 2\phi ^{n+1}-\phi ^n \right) ,2\phi ^{n+1}-\phi ^n \right) 
        +\frac{1}{2}\lVert 2q^{n+1}-q^n \rVert ^2-\left( B,1 \right) 
        $
        \end{comment}
        %\begin{comment}
            \begin{align*}
                {}&\frac{1}{4}\left[ \left( \mathcal{L}\phi ^{n+1},\phi ^{n+1} \right) +\left( \mathcal{L}\left( 2\phi ^{n+1}-\phi ^n \right) ,2\phi ^{n+1}-\phi ^n \right) \right] \!\\
                {}&+\frac{1}{2}\left( \lVert q^{n+1} \rVert ^2+\lVert 2q^{n+1}-q^n \rVert ^2 \right)\\
                {}&-\frac{1}{4}\left[ \left( \mathcal{L}\phi ^n,\phi ^n \right) +\left( \mathcal{L}\left( 2\phi ^n-\phi ^{n-1} \right) ,2\phi ^n-\phi ^{n-1} \right) \right]\\
                {}&-\frac{1}{2}\left( \lVert q^n \rVert ^2+\lVert 2q^n-q^{n-1} \rVert ^2 \right)\\
                \le{}& -\tau M\left( \mathcal{G}\mu ^{n+1},\mu ^{n+1} \right) \le 0.
            \end{align*}
        %\end{comment}
    \end{itemize}
    %with the corresponding modified energy $E(\phi^n)$.
    \begin{comment}
    \begin{itemize}
        \item First-order:
            $$
            E^{n+1}-E^n\le -\tau \left( \mathcal{G}\mu ^{n+\frac{1}{2}},\mu ^{n+\frac{1}{2}} \right) ,
            $$
        where $E^n=\frac{1}{2}\left( \mathcal{L}\phi ^n,\phi ^n \right) +\lVert q^n \rVert ^2-\left( B,1 \right)$, 
        \item Second-order:
        $$
        \bar{E}^{n+1}-\bar{E}^n\le -\tau \left( \mathcal{G}\mu ^{n+\frac{1}{2}},\mu ^{n+\frac{1}{2}} \right) .
        $$
        where $
        \bar{E}^n=\frac{1}{4}\left( \mathcal{L}\phi ^n,\phi ^n \right) +\frac{1}{4}\left( \mathcal{L}\left( 2\phi ^n-\phi ^{n-1} \right) ,2\phi ^n-\phi ^{n-1} \right) 
        +\frac{1}{2}\lVert q^n \rVert ^2+\frac{1}{2}\lVert 2q^n-q^{n-1} \rVert ^2-\left( B,1 \right) .
        $
    \end{itemize}
    \end{comment}
\end{lemma}

\subsection{ The REQ method}
To overcome the inconsistency issue of the modified energy and original energy at the discrete level,
Zhao \cite{ZHAO2021107331} proposes a relaxation technique based on the baseline IEQ method,
which called REQ method 
to overcome 
%this inconsistency issue of the modified energy and 
%the original energy after discretization. %we come up with the novel relaxed EQ schemes as follows
the numerical difference between $q^{n+1}$ and $Q(\phi^{n+1})$.
The semi-implicit relaxed IEQ Crank-Nicolson (REQ/CN) scheme is as follows.
% Now we consider the semi-implicit Crank-Nicolson (CN) time-marching method based on REQ approach.
%approach proposed in \cite{ZHAO2021107331}.
\begin{corollary}[REQ/CN scheme \cite{ZHAO2021107331}]
    For initial datas, we set $\phi ^0=\phi \left( \boldsymbol{x},0 \right)$ and $q^0=\sqrt{F\left( \phi ^0 \right) +C_0}$,
    then we update $\phi^{n+1}$, $q^{n+1}$ via the following two steps:
    \label{REQ/CN}
    \par
    \begin{itemize}
        \item[] $\bf{Step~1} : $ To obtain the intermediate solutions $(\phi^{n+1}, \hat{q}^{n+1})$ by the baseline IEQ method: 
        %\begin{equation}
            %\label{IEQ22}
            \begin{align*}
                {}&\frac{\phi ^{n+1}-\phi ^n}{\tau}=-M\mathcal{G}\mu ^{n+\frac{1}{2}},\\
                {}&\mu ^{n+\frac{1}{2}}=\mathcal{L}\frac{\phi ^{n+1}+\phi ^n}{2}+\frac{\hat{q}^{n+1}+q^n}{2}\bar{H}^{n+\frac{1}{2}},\\
                {}&\frac{\hat{q}^{n+1}-q^n}{\tau}=\frac{\bar{H}^{n+\frac{1}{2}}}{2}\frac{\phi ^{n+1}-\phi ^n}{\tau},
            \end{align*}   
        %\end{equation}
        where $\bar{H}^{n+\frac{1}{2}}=\frac{f\left( \bar{\phi}^{n+\frac{1}{2}} \right)}{Q\left( \bar{\phi}^{n+\frac{1}{2}} \right)}$, 
        $\bar{\phi}^{n+\frac{1}{2}}=\frac{3}{2}\phi ^n-\frac{1}{2}\phi ^{n-1}$.
        \item[] $\bf{Step~2} : $ To update the intermediate solution $\hat{q}^{n+1}$ with a relaxation step, we update $q^{n+1}$ as:
        \begin{equation}
            \label{REQ/CN-q}       
        \begin{aligned}
            {}&q^{n+1}=\xi ^{n+1}\hat{q}^{n+1}+\left( 1-\xi ^{n+1} \right) Q(\phi^{n+1}),\quad\xi ^{n+1}\in V_0\cap V_1,\\
            {}& V_0=\left\{ \xi \left| \xi \in \left[ 0,1 \right] \right. \right\}, \quad
            V_1=\left\{ \xi \left| a\xi ^2+b\xi +c \right. \le 0 \right\} ,
        \end{aligned}
    \end{equation}
        where the coefficients $a$, $b$ and $c$ are given by 
        \begin{align*}
            {}& a=\lVert \hat{q}^{n+1}-Q(\phi^{n+1}) \rVert ^2,\quad
            b=2\left( \hat{q}^{n+1}-Q\left( \phi ^{n+1} \right),Q\left( \phi ^{n+1} \right) \right), \\ 
            {}& c=-\tau \eta \left( \mathcal{G}\mu ^{n+\frac{1}{2}},\mu ^{n+\frac{1}{2}} \right)  +\lVert Q(\phi^{n+1}) \rVert ^2-\lVert \hat{q}^{n+1} \rVert ^2,\quad
            \eta\in[0,1].
        \end{align*}
    \end{itemize}
\end{corollary}
\begin{remark}[Optimal choice for $\xi^{n+1}$ \cite{ZHAO2021107331}]
    Here $\eta$ is an artificial parameter that can be assigned.
    If $a=0$, we set $\xi^{n+1}=0$ which means that $q^{n+1}=Q(\phi^{n+1})$, and if $a>0$, we take
    $\xi ^{n+1}=\max \left\{ 0,(-b-\sqrt{b^2-4ac})/(2a) \right\}$.
\end{remark}
\begin{lemma}[\cite{ZHAO2021107331}]
    \label{th-E_{1}^{n+1}}
    The REQ/CN scheme \ref{REQ/CN} is unconditionally energy stable in the sense that 
    $$
    \frac{1}{2}\left( \mathcal{L}\phi ^{n+1},\phi ^{n+1} \right) +\lVert q^{n+1} \rVert ^2-\frac{1}{2}\left( \mathcal{L}\phi ^n,\phi ^n \right) -\lVert q^n \rVert ^2=-\tau M \left( \mathcal{G}\mu ^{n+\frac{1}{2}},\mu ^{n+\frac{1}{2}} \right) \le 0.
    $$
    \begin{comment}
    $$
    E^{n+1}-E^n=-\tau \left( \mathcal{G}\mu ^{n+\frac{1}{2}},\mu ^{n+\frac{1}{2}} \right) ,
    $$
    where $E^n=\frac{1}{2}\left( \mathcal{L}\phi ^{n+1},\phi ^{n+1} \right) +\lVert q^{n+1} \rVert ^2-(B,1).$
    \end{comment}
\end{lemma}

\section{ The EOP-IEQ method}
The major issue is that the quxiliary variable $q(t_{n+1})$ is no longer equal to $Q(\phi(t_{n+1}))$ numerically. 
Thus the modified energy law \eqref{dE_{1}^{n+1}} does not necessarily satisfy the original energy law \eqref{dE}
at the discrete level.
To better restrain the inconsistency issue between $q(t_{n+1})$ and $Q(\phi(t_{n+1}))$ (that are supposed to be equal 
as introduced in \eqref{eq_q}), Liu et.al. propose a novel modified SAV method \cite{liu2023novel}, 
named  energy-optimized SAV (EOP-SAV) approach which is unconditionally energy stable with respect to a
modified energy that is closer to the original energy than the baseline SAV\cite{Shen2019} 
and relaxed SAV \cite{Zhang2023EfficientAA} approaches.
Inspired by this method, the EOP-IEQ method is considered in this paper,
which has the optimal correction on modified energy compared to the IEQ and REQ methods.
The main idea of the energy-optimized (EOP) technique is to correct the auxiliary variable $q^n$ 
from the perspective of original nonlinear energy $E_{\mathcal{N}}\left( \phi^n \right)$, 
so that the modified energy is closest to the original energy at the discrete level, 
and the numerical schemes preserve the energy dissipation law.

\subsection{The EOP-IEQ/CN scheme.}
Now we consider the following second-order Crank-Nicolson scheme based on the novel EOP-IEQ approach.
%set R0 = pE_{1}^{n+1}(φ0) + C, E_{1}^{n+1}(φn) =
%E_{1}^{n+1}(φn) + C and compute φn+1, Rn+1 via the following two steps:
\begin{corollary}[EOP-IEQ/CN scheme]
    For initial datas, we set $\phi ^0=\phi \left( \boldsymbol{x},0 \right)$ and $q^0=\sqrt{F\left( \phi ^0 \right) +C_0}$,
    then we update $\phi^{n+1}$, $q^{n+1}$ via the following two steps:
    \label{EOP/CN}
    \par
    \begin{itemize}
        \item[] $\bf{Step~1} : $ To obtain the intermediate solutions $(\phi^{n+1},\hat{q}^{n+1})$ by the baseline IEQ method: 
        \begin{equation} 
            \label{EOP/CN1}
            \begin{aligned}
                {}&\frac{\phi ^{n+1}-\phi ^n}{\tau}=-M\mathcal{G}\mu ^{n+\frac{1}{2}},\\
                {}&\mu ^{n+\frac{1}{2}}=\mathcal{L}\frac{\phi ^{n+1}+\phi ^n}{2}+\frac{\hat{q}^{n+1}+q^n}{2}\bar{H}^{n+\frac{1}{2}},\\
                {}&\frac{\hat{q}^{n+1}-q^n}{\tau}=\frac{\bar{H}^{n+\frac{1}{2}}}{2}\frac{\phi ^{n+1}-\phi ^n}{\tau},
            \end{aligned}   
        \end{equation}
        where $\bar{H}^{n+\frac{1}{2}}=\frac{f\left( \bar{\phi}^{n+\frac{1}{2}} \right)}{Q\left( \bar{\phi}^{n+\frac{1}{2}} \right)}$, 
        $\bar{\phi}^{n+\frac{1}{2}}=\frac{3}{2}\phi ^n-\frac{1}{2}\phi ^{n-1}$.
        \item[] $\bf{Step~2} : $ To optimize the intermediate solution $\hat{q}^{n+1}$ with an energy-optimized step,
         we update the auxiliary variable $q^{n+1}$ as
        \begin{equation}
            \label{EOP/CN2}
            q^{n+1}=\lambda ^{n+1}Q(\phi^{n+1}),\quad \lambda ^{n+1}=\min \left\{ 1,\sqrt{\frac{E_{2}^{n+1}}{E_{1}^{n+1}}} \right\}, 
        \end{equation}
        where the coefficients $E_{1}^{n+1}$ and $E_{2}^{n+1}$ are given by
        \begin{align*}
            E_{1}^{n+1}=\lVert Q\left( \phi ^{n+1} \right) \rVert ^2,\quad 
            E_{2}^{n+1}=\frac{1}{2}\left( \mathcal{L}\phi ^n,\phi ^n \right) +\lVert q^n \rVert ^2-\frac{1}{2}\left( \mathcal{L}\phi ^{n+1},\phi ^{n+1} \right) .  
        \end{align*}
    \end{itemize}
\end{corollary}
\begin{remark}
    In the computational simulation, the value of the intermediate variable $\hat{q}^{n+1}$ in $\bf{Step~1}$ does not need to be solved, 
    it is only used in the following numerical analysis.
    According to the energy-optimized technique in $\bf{Step~2}$, the numerical solutions of $q^{n+1}$ are always non-negative.
    Furthermore, the elements discretized by $q^{n+1}$ in space are all non-negative at the fully discrete level.
\end{remark}
\begin{theorem}\label{th-E_{2}^{n+1}}
    The energy-optimized technique of $q^{n+1}$ \eqref{EOP/CN2} in $\bf{Step~2}$ is an optimal choice to correct
    %the auxiliary variable $\hat{q}^{n+1}$ of 
    the modified energy dissipative law of the baseline IEQ method \cite{YANG2016294}.
\end{theorem}
\begin{proof}
    Taking the inner products of \eqref{EOP/CN1} with $\mu ^{n+\frac{1}{2}}$, $\frac{\phi ^{n+1}-\phi ^n}{\tau}$ and $\hat{q}^{n+1}+q^n$
    respectively, we obtain
    \begin{comment}
    $$
    \frac{1}{2}\left( \mathcal{L}\phi ^{n+1},\phi ^{n+1} \right) +\lVert \hat{q}^{n+1} \rVert ^2-\frac{1}{2}\left( \mathcal{L}\phi ^n,\phi ^n \right) -\lVert q^n \rVert ^2=-\tau \left( \mathcal{G}\mu ^{n+\frac{1}{2}},\mu ^{n+\frac{1}{2}} \right)\le 0 .
    $$ 
    \end{comment}
    \begin{align}
        \label{qCN}
        \lVert \hat{q}^{n+1} \rVert ^2=\frac{1}{2}\left( \mathcal{L}\phi ^n,\phi ^n \right) +\lVert q^n \rVert ^2-\frac{1}{2}\left( \mathcal{L}\phi ^{n+1},\phi ^{n+1} \right) -\tau M\left( \mathcal{G}\mu ^{n+\frac{1}{2}},\mu ^{n+\frac{1}{2}} \right) .
    \end{align}
    %In terms of energy dissipation law, if we correct $q^{n+1}$, at least the following inequality holds
    As far as the energy dissipation law is concerned, if we correct $q^{n+1}$, at least the following inequality holds
    \begin{align}
        \label{qEdiss}
        \frac{1}{2}\left( \mathcal{L}\phi ^{n+1},\phi ^{n+1} \right) +\lVert q^{n+1} \rVert ^2-\frac{1}{2}\left( \mathcal{L}\phi ^n,\phi ^n \right) -\lVert q^n \rVert ^2\le 0.
    \end{align}
    Then we have the following inequality
    \begin{align}
        \label{qE}
        0\le \lVert q^{n+1} \rVert ^2\le \frac{1}{2}\left( \mathcal{L}\phi ^n,\phi ^n \right) +\lVert q^n \rVert ^2-\frac{1}{2}\left( \mathcal{L}\phi ^{n+1},\phi ^{n+1} \right).
    \end{align}
    %From \eqref{qE}, we know $E_{2}^{n+1}\ge 0$. 
    Let $E_{1}^{n+1}=\lVert Q\left( \phi ^{n+1} \right) \rVert ^2$, 
    $ E_{2}^{n+1} = \frac{1}{2}\left( \mathcal{L}\phi ^n,\phi ^n \right) +\lVert q^n \rVert ^2-\frac{1}{2}\left( \mathcal{L}\phi ^{n+1},\phi ^{n+1} \right)$,
    we have
    \begin{enumerate}[(1)]
        \item[(i)] $E_{1}^{n+1}\le E_{2}^{n+1}$. We update $q^{n+1}$ by $q^{n+1}  =Q\left( \phi ^{n+1} \right) $
        and \eqref{qEdiss} is valid obviously.
        It means that the discrete modified energy is totally equal to the original energy, i.e.
        $$
        \frac{1}{2}\left( \mathcal{L}\phi ^{n+1},\phi ^{n+1} \right) +\lVert q^{n+1} \rVert ^2=\frac{1}{2}\left( \mathcal{L}\phi ^{n+1},\phi ^{n+1} \right) +\lVert Q\left( \phi ^{n+1} \right) \rVert ^2.
        $$
        \item[(ii)] $E_{1}^{n+1}> E_{2}^{n+1}$. 
        Combining \eqref{qCN} and \eqref{qE}, we see that
        $0\le E_{2}^{n+1}=\lVert \hat{q}^{n+1} \rVert ^2+\tau M \left( \mathcal{G}\mu ^{n+\frac{1}{2}},\mu ^{n+\frac{1}{2}} \right)$, 
        %by \eqref{qCN} and \eqref{qE},
        which means $\lVert \hat{q}^{n+1} \rVert ^2\le E_{2}^{n+1}<\lVert Q\left( \phi ^{n+1} \right) \rVert ^2$.
        In terms of the closest approximation to the original energy,
        %$E_{2}^{n+1}$ is the maximum that satisfies the energy dissipative law \eqref{qEdiss}, 
        $E_{2}^{n+1}$ is the closest value to $\lVert Q\left( \phi ^{n+1} \right) \rVert ^2$
        and satisfies the energy dissipative law \eqref{qEdiss}.
        So we update $q^{n+1}$ by $q^{n+1}=\sqrt{E_{2}^{n+1}/E_{1}^{n+1}}Q\left( \phi ^{n+1} \right)$,
        it means that $\lVert q^{n+1} \rVert ^2=E_{2}^{n+1}$ and \eqref{qEdiss} is valid.
    \end{enumerate}
\end{proof}
\begin{theorem}\label{th-e3}
    The EOP-IEQ/CN scheme \ref{EOP/CN} is unconditionally energy stable in the sense that 
    \begin{equation}
        \label{CN1}
    \begin{aligned}
        {}&\frac{1}{2}\left( \mathcal{L}\phi ^{n+1},\phi ^{n+1} \right) 
        +\lVert q^{n+1} \rVert ^2-\frac{1}{2}\left( \mathcal{L}\phi ^n,\phi ^n \right) 
        -\lVert q^n \rVert ^2\\
        ={}&-\tau M\left( \mathcal{G}\mu ^{n+\frac{1}{2}},\mu ^{n+\frac{1}{2}} \right)\le 0 ,
    \end{aligned}
    \end{equation}
    and if $E_{1}^{n+1}\le E_{2}^{n+1}$, we further have the following original dissipative law
    $$
    \frac{1}{2}\left( \mathcal{L}\phi ^{n+1},\phi ^{n+1} \right) +\lVert Q\left( \phi ^{n+1} \right) \rVert ^2-\frac{1}{2}\left( \mathcal{L}\phi ^n,\phi ^n \right) -\lVert Q\left( \phi ^n \right) \rVert ^2\le 0.
    $$
    %under the condition of $E_{1}^{n+1}\le E_{2}^{n+1}$.
\end{theorem}
\begin{proof}
    From the equation \eqref{EOP/CN2} in $\bf{Step~2}$, we obtain $\lVert q^{n+1} \rVert ^2\le E_{2}^{n+1}$,
    so \eqref{CN1} is obviously valid.
    %it means \eqref{CN1} is valid.
    %$$
   % \frac{1}{2}\left( \mathcal{L}\phi ^{n+1},\phi ^{n+1} \right) +\lVert q^{n+1} \rVert ^2-\frac{1}{2}\left( \mathcal{L}\phi ^n,\phi ^n \right) -\lVert q^n \rVert ^2=-\tau \left( \mathcal{G}\mu ^{n+\frac{1}{2}},\mu ^{n+\frac{1}{2}} \right)\le 0 .
    %$$
    According to $\bf{Step~2}$, we also obtain $\lVert q^n \rVert ^2\le E_{1}^{n}$. 
    %$\lVert q^{n+1} \rVert ^2\le E_{1}^{n+1}$.
    So we have  
    \begin{align}
        \label{CN2}
        \frac{1}{2}\left( \mathcal{L}\phi ^{n},\phi ^{n} \right) +\lVert q^{n} \rVert ^2\le \frac{1}{2}\left( \mathcal{L}\phi ^{n},\phi ^{n} \right) +\lVert Q\left( \phi ^{n} \right) \rVert ^2, \quad n\ge 0.
    \end{align}
    If $E_{1}^{n+1}\le E_{2}^{n+1}$, we obtain $\lVert q^{n+1} \rVert ^2=E_{1}^{n+1}$, which means
    \begin{align}
        \label{CN3}
    \frac{1}{2}\left( \mathcal{L}\phi ^{n+1},\phi ^{n+1} \right) +\lVert Q\left( \phi ^{n+1} \right) \rVert ^2=\frac{1}{2}\left( \mathcal{L}\phi ^{n+1},\phi ^{n+1} \right) +\lVert q^{n+1} \rVert ^2.
    \end{align}
    Combining \eqref{CN1}, \eqref{CN2} and \eqref{CN3}, we obtain 
    $$
    \frac{1}{2}\left( \mathcal{L}\phi ^{n+1},\phi ^{n+1} \right) +\lVert Q\left( \phi ^{n+1} \right) \rVert ^2\le \frac{1}{2}\left( \mathcal{L}\phi ^n,\phi ^n \right) +\lVert Q\left( \phi ^n \right) \rVert ^2.
    $$
    The proof of the theorem is now complete.
\end{proof}
\begin{theorem}\label{th-e4}
    The EOP-IEQ/CN scheme \ref{EOP/CN} is a special case of
    %is energy-optimized
    % equivalent to 
    the REQ/CN scheme \ref{REQ/CN}
    under the condition of the artificial parameter $\eta=1$ which means that the maximum relaxation is introduced.
\end{theorem}
\begin{proof}
    Suppose that the intermediate solutions $(\phi^{n+1}, \hat{q}^{n+1})$ are calculated 
    from the baseline IEQ/CN scheme \eqref{EOP/CN1}, 
    we show that the conclusions of the two schemes are consistent with $\eta=1$,.

    If $\lVert \hat{q}^{n+1}-Q\left( \phi ^{n+1} \right) \rVert=0$, 
    we know $\lVert \hat{q}^{n+1} \rVert =\lVert Q\left( \phi ^{n+1} \right) \rVert$. 
    For the the relaxation step \eqref{REQ/CN-q} of the REQ/CN scheme \ref{REQ/CN}, 
    we show that $\xi^{n+1}=0.$   Since $a=0$, it follows that $\xi^{n+1}=0$.
    For the energy-optimized technique \eqref{EOP/CN2} of the EOP-IEQ/CN scheme \ref{EOP/CN},
    we show that $\lambda^{n+1}=1$.
    According to the definition of $E_{2}^{n+1}$ and \eqref{qCN}, we have 
    $E_{1}^{n+1}=\lVert Q\left( \phi ^{n+1} \right) \rVert ^2\le 
    \lVert \hat{q}^{n+1} \rVert ^2+\tau M ( \mathcal{G}\mu ^{n+\frac{1}{2}},\mu ^{n+\frac{1}{2}} ) =E_{2}^{n+1}$,
    so $\lambda ^{n+1}=\min \left\{ 1,\sqrt{E_{2}^{n+1}/E_{1}^{n+1}} \right\} =1$.

    If $\lVert \hat{q}^{n+1}-Q\left( \phi ^{n+1} \right) \rVert> 0$, 
    we know $a>0$, then we analyze the values of $b$ and $c$.
    \begin{enumerate}[(1)]
        \item[(i)] $b>0$.
            For the relaxation step \eqref{REQ/CN-q}, 
            we have $(-b-\sqrt{b^2-4ac})/(2a)<0$ and 
            $\xi ^{n+1}=0$.
            %$\xi ^{n+1}=\max \left\{ (-b-\sqrt{b^2-4ac})/(2a),0 \right\} =0$.
            %$\frac{-b-\sqrt{b^2-4ac}}{2a}<0$, so $\xi^{n+1}=0$, 
            According to the energy-optimized technique \eqref{EOP/CN}, we obtain
            $
                \lVert Q\left( \phi ^{n+1} \right) \rVert ^2<\left( \hat{q}^{n+1},Q\left( \phi ^{n+1} \right) \right) \le \left( \lVert \hat{q}^{n+1} \rVert ^2+\lVert Q\left( \phi ^{n+1} \right) \rVert ^2 \right)/2 .
            $
            Combining the above equality with the definition of $E_{2}^{n+1}$,
            %$E_{2}^{n+1}=\lVert \hat{q}^{n+1} \rVert ^2+\tau ( \mathcal{G}\mu ^{n+\frac{1}{2}},\mu ^{n+\frac{1}{2}} )$,
            we obtain
            $
                E^{n+1}_1=\lVert Q\left( \phi ^{n+1} \right) \rVert ^2<\lVert \hat{q}^{n+1} \rVert ^2\le E^{n+1}_2,
            $
            so $\lambda ^{n+1}=1$.
            %$\lambda ^{n+1}=\min \left\{ 1,\sqrt{E_{2}^{n+1}/E_{1}^{n+1}} \right\} =1$.
        \item[(ii)] $b\le 0$, $c\le 0$.
            For the relaxation step \eqref{REQ/CN-q}, we know that 
            $\xi ^{n+1}=\max \left\{ 0,\left( -b-\sqrt{b^2-4ac} \right) /\left( 2a \right) \right\}=0$,
            %$(-b-\sqrt{b^2-4ac})/(2a)\le 0$, so $\xi^{n+1}=0$, 
            and for the energy-optimized step \eqref{EOP/CN}, we obtain
            $
            %E_{1}^{n+1}=
            \lVert Q\left( \phi ^{n+1} \right) \rVert ^2\le \lVert \hat{q}^{n+1} \rVert ^2+\tau M ( \mathcal{G}\mu ^{n+\frac{1}{2}},\mu ^{n+\frac{1}{2}} ) 
            %=E_{2}^{n+1}
            $ 
            by $c\le 0$,
            %$c=-\tau ( \mathcal{G}\mu ^{n+\frac{1}{2}},\mu ^{n+\frac{1}{2}} )  +\lVert Q(\phi^{n+1}) \rVert ^2-\lVert \hat{q}^{n+1} \rVert ^2\le 0$,
            which means $E_{1}^{n+1}\le E_{2}^{n+1}$,
            so $\lambda ^{n+1}=1$. %\min \left\{ 1,\sqrt{\frac{E_{2}^{n+1}}{E_{1}^{n+1}}} \right\} =1$.
        \item[(iii)] $b\le 0$, $c>0$.
            For the REQ/CN scheme \ref{REQ/CN-q}, we obtain 
            %$(-b-\sqrt{b^2-4ac})/(2a)>0$, so 
            $\xi ^{n+1}=(-b-\sqrt{b^2-4ac})/(2a)$.
            %According to $a(\xi^{n+1})^2+b\xi^{n+1}+c=0$, we have
            %$
            %\lVert q^{n+1} \rVert ^2=\lVert \hat{q}^{n+1} \rVert ^2+\tau ( \mathcal{G}\mu ^{n+\frac{1}{2}},\mu ^{n+\frac{1}{2}} ) =E_{2}^{n+1}.
            %$
            For the EOP-IEQ/CN scheme \ref{EOP/CN}, we have
            %$
            %%E_{1}^{n+1}=
            %\lVert Q\left( \phi ^{n+1} \right) \rVert ^2>\lVert \hat{q}^{n+1} \rVert ^2+\tau ( \mathcal{G}\mu ^{n+\frac{1}{2}},\mu ^{n+\frac{1}{2}} ) 
            %%=E_{2}^{n+1} 
            %$
            $E_{2}^{n+1}<E_{1}^{n+1}$ by $c>0$, %which means $E_{2}^{n+1}<E_{1}^{n+1}$, 
            %$c=-\tau ( \mathcal{G}\mu ^{n+\frac{1}{2}},\mu ^{n+\frac{1}{2}} ) +\lVert Q(\phi^{n+1}) \rVert ^2-\lVert \hat{q}^{n+1} \rVert ^2> 0$,
            so $\lambda ^{n+1} =\sqrt{E_{2}^{n+1}/E_{1}^{n+1}}$.
    \end{enumerate}
    In conclusion, when $\xi^{n+1}=0$, we have $\lambda^{n+1}=1$, so we update $q^{n+1}$ by $q^{n+1}=Q(\phi^{n+1})$; 
    and when $\xi^{n+1}=(-b-\sqrt{b^2-4ac})/(2a)$, 
    we obtain $\lambda^{n+1}=\sqrt{E_{2}^{n+1}/E_{1}^{n+1}}$, then we get 
    \begin{align*}
        q^{n+1}{}&=\frac{-b-\sqrt{b^2-4ac}}{2a}\hat{q}^{n+1}+\left( 1-\frac{-b-\sqrt{b^2-4ac}}{2a} \right) Q\left( \phi ^{n+1} \right)\\
        {}& =\sqrt{\frac{E_{2}^{n+1}}{E_{1}^{n+1}}}Q\left( \phi ^{n+1} \right) .
    \end{align*}
    So the REQ/CN scheme is energy-optimized
    % equivalent to the REQ/CN scheme \ref{REQ/CN}
    with the artificial parameter $\eta=1$.
    %The proof is completed.
\end{proof}
\begin{remark}
    Theorem \ref{th-e4} shows that when the artificial parameter $\eta=1$, the relaxation step 
    achieves energy optimization, 
    which means that it agrees with the results obtained by the energy-optimized step,
    and when $\eta\in [0,1)$, the energy-optimized technique is always equal to or better than the relaxation technique.
\end{remark}
%\subsection{ consistency estimate }
Now we establish the consistency estimate for the EOP-IEQ/CN scheme \ref{EOP/CN}. 
The consistency errors $\varepsilon _{1}^{n+1}$, $\varepsilon _{2}^{n+1}$, $\varepsilon _{3}^{n+1}$
for the EOP-IEQ/CN method are determined by the following:
\begin{subequations}
    \label{trun}
    \begin{align}
		{}&\frac{\phi _{\star}^{n+1}-\phi _{\star}^{n}}{\tau}=-M\mathcal{G}\left( \mathcal{L}\phi _{\star}^{n+\frac{1}{2}}+H\left( \phi _{\star}^{n+\frac{1}{2}} \right) \hat{q}_{\star}^{n+\frac{1}{2}} \right) +\varepsilon _{1}^{n+1},\\
		{}&\frac{\hat{q}_{\star}^{n+1}-q_{\star}^{n}}{\tau}=\frac{H\left( \phi _{\star}^{n+\frac{1}{2}} \right)}{2}\left( \frac{\phi _{\star}^{n+1}-\phi _{\star}^{n}}{\tau} \right) +\varepsilon _{2}^{n+1},\\
		{}&q_{\star}^{n+1}=\lambda ^{n+1}Q\left( \phi _{\star}^{n+1} \right)+\varepsilon _{3}^{n+1} ,\quad \lambda ^{n+1}=\min \left\{ 1,\sqrt{\frac{E^{n+1}_2}{E^{n+1}_1}} \right\},
	\end{align}
\end{subequations}
where $\phi _{\star}^{n+1}=\phi \left( t_{n+1} \right) $, $ q_{\star}^{n+1}=q\left( t_{n+1} \right) $, $ Q_{\star}^{n+1}=Q\left( \phi(t_{n+1}) \right)$.\par
Suppose that the analytical solutions ($\phi,q$) of the equivalent gradient flow model \eqref{IEQ} 
possess the following regularity condition
    \begin{equation}
        \label{zheng}
        \left\{ \begin{array}{l}
            \phi \in L^{\infty}\left( 0,T;H^2\left( \Omega \right) \right) ,\quad q\in L^{\infty}\left( 0,T;L^{\infty}\left( \Omega \right) \right) ,\\
            \phi _t\in L^{\infty}\left( 0,T;L^{\infty}\left( \Omega \right) \right) ,\quad q_{tt},\phi _{tt}\in L^2\left( 0,T;L^2\left( \Omega \right) \right).
        \end{array} \right. 
    \end{equation}
We first give the following lemma to establish the quantitative relation between the $L^2$ norm 
of $ H( \phi _{\star}^{n+1}) -H( \phi _{\star}^{n} )$ 
and $\phi _{\star}^{n+1}-\phi _{\star}^{n}$ under the regularity condition \eqref{zheng}.
    \begin{lemma}[\cite{Yang2020ConvergenceAF}]
        \label{HH}
        Suppose that
        \begin{enumerate}[(1)]
            \item[(1)] $F(x)$ uniformly bounded from below: $F(x)>-C_*$, $x\in(-\infty,+\infty)$;
            \item[(2)] $F(x) \in C^2(-\infty,+\infty)$;
            \item[(3)] there exists a positive constant $C_{\star}$ such that %$\lVert \phi \rVert _{L^{\infty}\left( 0,T;L^{\infty}\left( \Omega \right) \right)}\le C_0$.
            $\underset{0\le n\le N}{\max}\left( \lVert \phi^n_{\star} \rVert _{L^{\infty}} \right) \le C_{\star}$.
        \end{enumerate}
    We have
    \begin{align*}
    \lVert H\left( \phi _{\star}^{n+1} \right) -H\left( \phi _{\star}^{n} \right) \rVert \le C_1\lVert \phi _{\star}^{n+1}-\phi _{\star}^{n} \rVert ,\quad n\ge 0,
    \end{align*}
    %$$
    %\lVert H\left( \phi _{\star}^{n+1} \right) -H\left( \frac{3}{2}\phi _{\star}^{n}-\frac{1}{2}\phi _{\star}^{n-1} \right) \rVert \le C_1\lVert \phi _{\star}^{n+1}-\frac{3}{2}\phi _{\star}^{n}+\frac{1}{2}\phi _{\star}^{n-1} \rVert ,
    %$$
     where $C_1$ dependens only on $C_*$, $C_0$, $C_{\star}$.
    \end{lemma}
\begin{theorem}
	\label{AC-err}
    We consider $\mathcal{G} = I$ for the $L^2$ gradient flow for simplicity, and the the analytical solutions ($\phi,q$)
    satisfie the regularity condition \eqref{zheng}, then the following consistency estimate holds for the system \eqref{trun}:
$$
\lVert \varepsilon _{1}^{n+1} \rVert +\lVert \varepsilon _{2}^{n+1} \rVert +\lVert \varepsilon _{3}^{n+1} \rVert \le C\tau^2, 
$$
where $C$ is independent of $\tau$.
\end{theorem}
\begin{proof}
    Let $R_{\phi}^{n+1}$ be the truncation error defined by
    \begin{equation}
        \label{err1}
        \begin{aligned}
        R_{\phi}^{n+1}
        ={}&\frac{\phi _{\star}^{n+1}-\phi _{\star}^{n}}{\tau}-\frac{\partial \phi \left( t_{n+\frac{1}{2}} \right)}{\partial t}\\
        ={}&\frac{1}{2\tau}\left( \int_{t_n}^{t_{n+\frac{1}{2}}}{\left( t_n-s \right) ^2\frac{\partial ^2\phi \left( s \right)}{\partial t^2}}ds+\int_{t_{n+\frac{1}{2}}}^{t_{n+1}}{\left( t_{n+1}-s \right) ^2}\frac{\partial ^2\phi \left( s \right)}{\partial t^2}ds \right) .
        \end{aligned}
    \end{equation}
   % $$
   % R_{\phi}^{n+1}=\frac{\phi _{\star}^{n+1}-\phi _{\star}^{n}}{\tau}-\frac{\partial \phi \left( t_{n+\frac{1}{2}} \right)}{\partial t}=\frac{1}{2\tau}\left( \int_{t_n}^{t_{n+\frac{1}{2}}}{\left( t_n-s \right) ^2\frac{\partial ^2\phi \left( s \right)}{\partial t^2}}ds+\int_{t_{n+\frac{1}{2}}}^{t_{n+1}}{\left( t_{n+1}-s \right) ^2}\frac{\partial ^2\phi \left( s \right)}{\partial t^2}ds \right) .
   % $$
    Since 
    \begin{equation}
        \label{err2}
        \begin{aligned}
            {}&\frac{\phi _{\star}^{n+1}+\phi _{\star}^{n}}{2}-\phi _{\star}^{n+\frac{1}{2}}=\frac{1}{2}\left( \int_{t_n}^{t_{n+\frac{1}{2}}}{\left( s-t_n \right) \frac{\partial ^2\phi}{\partial t^2}\left( s \right)}ds+\left( t_{n+1}-s \right) \frac{\partial ^2\phi}{\partial t^2}\left( s \right) ds \right) ,\\
            {}&\frac{\hat{q}_{\star}^{n+1}+\phi _{\star}^{n}}{2}-q_{\star}^{n+\frac{1}{2}}=\frac{1}{2}\left( \int_{t_n}^{t_{n+\frac{1}{2}}}{\left( s-t_n \right) \frac{\partial ^2q}{\partial t^2}\left( s \right)}ds+\left( t_{n+1}-s \right) \frac{\partial ^2q}{\partial t^2}\left( s \right) ds \right) .
        \end{aligned}
    \end{equation}
    From the Lemma \ref{HH}, we obtain
    \begin{align}
        \label{err3}
        \lVert H\left( \phi _{\star}^{n+1} \right) -H\left( \frac{3}{2}\phi _{\star}^{n}-\frac{1}{2}\phi _{\star}^{n-1} \right) \rVert \le C_1\lVert \phi _{\star}^{n+1}-\frac{3}{2}\phi _{\star}^{n}+\frac{1}{2}\phi _{\star}^{n-1} \rVert \le C_1C_0\tau^2.
    \end{align}
    By the three inequalities \eqref{err1}-\eqref{err3} and triangle inequality, we have 
    \begin{align*}
        \lVert \varepsilon _{1}^{n+1} \rVert 
        \le{}& \lVert R_{\phi}^{n+1} \rVert +\varepsilon ^2\lVert \Delta \left( \frac{\phi _{\star}^{n+1}+\phi _{\star}^{n}}{2}-\phi _{\star}^{n+\frac{1}{2}} \right) \rVert \\
        {}&+\lVert \frac{\hat{q}^{n+1}_{\star}+q^n_{\star}}{2}H\left( \frac{3}{2}\phi _{\star}^{n}-\frac{1}{2}\phi _{\star}^{n-1} \right) -q_{\star}^{n+\frac{1}{2}}H\left( \phi _{\star}^{n+\frac{1}{2}} \right) \rVert \\
        \le{}& C\tau^2.
    \end{align*}
    Similarly we can obtain that $\lVert \varepsilon _{2}^{n+1} \rVert\le C\tau^2$.
    By Theorem \ref{th-e4}, we obtain 
    $q_{\star}^{n+1}=\xi ^{n+1}\hat{q}^{n+1}+\left( 1-\xi ^{n+1} \right) Q\left( \phi _{\star}^{n+1} \right) 
    =Q\left( \phi _{\star}^{n+1} \right) +O\left( \tau ^2 \right),\forall \xi ^{n+1}\in \left[ 0,1 \right]$.
\end{proof}
\begin{remark}[Order of Accuracy \cite{ZHAO2021107331}]
    The proposed EOP-IEQ/CN scheme \ref{EOP/CN} is second-order accurate in time. 
    Notice the fact, $\hat{q}^{n+1}=q(t_{n+1})+O(\tau^2)$ and $Q(\phi^{n+1})=q(t_{n+1})+O(\tau^2)$.
    Thus for $\eta=1$, $q^{n+1}=\xi^{n+1}\hat{q}^{n+1}+(1-\xi^{n+1})Q(\phi^{n+1})=q(t_{n+1})+O(\tau^2), \forall \xi^{n+1}\in[0,1]$,
    which means the energy-optimized technique does not affect the order of accuracy for the baseline IEQ method.
\end{remark}

\subsection{The EOP-IEQ/BDF$k$ schemes}
If we use the backward Euler formulas for the time discretization, 
we will have the semi-implicit EOP-IEQ/BDF$k$ $(k=1,2)$ schemes as below.
\begin{corollary}[$k$th-order EOP-IEQ/BDF$k$ Scheme]
\label{BDFk}
For initial datas, we set $\phi ^0=\phi \left( \boldsymbol{x},0 \right)$ and $q^0=\sqrt{F\left( \phi ^0 \right) +C_0}$,
    then we update $\phi^{n+1}$, $q^{n+1}$ via the following two steps:
\begin{itemize}
	\item[] $\bf{Step~1} : $
	To calculate the intermediate solutions $(\phi^{n+1}, \hat{q}^{n+1})$ by the following semi-implicit IEQ/BDF$k$ schemes:
	%\begin{equation}
		%\label{IEQ22}
		\begin{align*}
            {}&\frac{\alpha _k\phi ^{n+1}-A_k\left( \phi ^n \right)}{\tau}=-M\mathcal{G}\mu ^{n+1},\\
            {}&\mu ^{n+1}=\mathcal{L}\phi ^{n+1}+\hat{q}^{n+1}H\left( B_k\left( \phi ^n \right) \right) ,\\
            {}&\frac{\alpha _k\hat{q}^{n+1}-A_k\left( q^n \right)}{\tau}=\frac{H\left( B_k\left( \phi ^n \right) \right)}{2}\frac{\alpha _k\phi ^{n+1}-A_k\left( \phi ^n \right)}{\tau},
        \end{align*}   
	%\end{equation}
	where $H\left( B_k\left( \phi ^n \right) \right) =\frac{f\left( B_k\left( \phi ^n \right) \right)}{Q\left( B_k\left( \phi ^n \right) \right)}$.
	\item[] $\bf{Step~2} : $
	To update the auxiliary variable $q^{n+1}$ via an energy-optimized step as:
        \begin{itemize}
            \item[] $\bullet$ \rm{ First-order EOP/BDF$ 1 $}. %\text{(\uppercase\expandafter{\romannumeral1})} 
            \begin{equation}
                \label{EOP/BDF1}
                q^{n+1}=\lambda ^{n+1}Q(\phi^{n+1}),\quad \lambda ^{n+1}=\min \left\{ 1,\sqrt{\frac{E_{2}^{n+1}}{E_{1}^{n+1}}} \right\}. 
            \end{equation}
            where the coefficients $E_{1}^{n+1}$ and $E_{2}^{n+1}$ are given by
            \begin{align*}
                E_{1}^{n+1}\!=\!\lVert Q( \phi ^{n+1} ) \rVert ^2,\quad 
                E_{2}^{n+1}\!=\!\frac{1}{2}( \mathcal{L}\phi ^n,\phi ^n ) \!+\!\lVert q^n \rVert ^2
                \!-\!\frac{1}{2}( \mathcal{L}\phi ^{n+1},\phi ^{n+1} ) .  
            \end{align*}
            \item[] $\bullet$ \rm{ Second-order EOP/BDF$ 2 $}.
            \begin{align}
                \label{EOP/BDF2}
                q^{n+1}=\lambda ^{n+1}\left( Q\left( \phi ^{n+1} \right) \!-\!\frac{2}{5}q^n \right)\!+\!\frac{2}{5}q^n ,\quad 
                \lambda ^{n+1}=\min \left\{ 1,\sqrt{\frac{\bar{E}_{2}^{n+1}}{\bar{E}_{1}^{n+1}}} \right\} ,
            \end{align}
            where the coefficients are given by
            \begin{align*}
            \bar{E}_{1}^{n+1}={}&\lVert Q\left( \phi ^{n+1} \right) -\frac{2}{5}q^n \rVert ^2,\\
            \bar{E}_{2}^{n+1}={}&\frac{1}{10}\left( \mathcal{L}\phi ^n,\phi ^n \right) +\frac{1}{10}\left( \mathcal{L}\left( 2\phi ^n-\phi ^{n-1} \right) ,2\phi ^n-\phi ^{n-1} \right)\\
                    {}&-\frac{1}{10}\left( \mathcal{L}\phi ^{n+1},\phi ^{n+1} \right) -\frac{1}{10}\left( \mathcal{L}\left( 2\phi ^{n+1}-\phi ^n \right) ,2\phi ^{n+1}-\phi ^n \right)\\
                    {}&+\frac{4}{25}\lVert q^n \rVert ^2+\frac{1}{5}\lVert 2q^n-q^{n-1} \rVert ^2.
            \end{align*}
            \begin{comment}
            \begin{align*}
                \bar{E}_{1}^{n+1}={}&\lVert Q\left( \phi ^{n+1} \right) -\frac{2}{5}q^n \rVert ^2,\\
                \bar{E}_{2}^{n+1}={}&\frac{1}{10}\left( \mathcal{L}\phi ^n,\phi ^n \right) -\frac{1}{10}\left( \mathcal{L}\phi ^{n+1},\phi ^{n+1} \right) +\frac{1}{10}\left( \mathcal{L}\left( 2\phi ^n-\phi ^{n-1} \right) ,2\phi ^n-\phi ^{n-1} \right) \\
                {}&-\frac{1}{10}\left( \mathcal{L}\left( 2\phi ^{n+1}-\phi ^n \right) ,2\phi ^{n+1}-\phi ^n \right) +\frac{4}{25}\lVert q^n \rVert ^2+\frac{1}{5}\lVert 2q^n-q^{n-1} \rVert ^2.
            \end{align*}
            \end{comment}
        \end{itemize}
\end{itemize}
\end{corollary}

\begin{theorem}
    The EOP-IEQ/BDF$k$ $ ( k = 1, 2 ) $ schemes \ref{BDFk} are unconditionally energy stable.
\end{theorem}
\begin{proof} For $k=1, 2$, we have the following analysis process respectively.
        \begin{itemize}
            \item[] $\bullet$ First-order EOP-IEQ/BDF$1$ scheme.
            From \eqref{EOP/BDF1},  we obtain
            \begin{align*}
                \lVert q^{n+1} \rVert ^2=\left( \lambda ^{n+1} \right) ^2\lVert Q\left( \phi ^{n+1} \right) \rVert ^2\le \left( \sqrt{\frac{E_{2}^{n+1}}{E_{1}^{n+1}}} \right) ^2 E_{1}^{n+1}=E_{2}^{n+1},
            \end{align*}
            According to the definition of $E_{2}^{n+1}$ and above inequation, we have
            \begin{align*}
                {}&\frac{1}{2}\left( \mathcal{L}\phi ^{n+1},\phi ^{n+1} \right) +\lVert q^{n+1} \rVert ^2-\frac{1}{2}\left( \mathcal{L}\phi ^n,\phi ^n \right) -\lVert q^n \rVert ^2\\
                \le{}& -\tau M \left( \mathcal{G}\mu ^{n+1},\mu ^{n+1} \right) \le 0.
            \end{align*}
            \item[] $\bullet$ Second-order EOP-IEQ/BDF$2$ scheme.
            From \eqref{EOP/BDF2},  we have
            $$
            \lVert q^{n+1}-\frac{2}{5}q^n \rVert ^2=\left( \lambda ^{n+1} \right) ^2\lVert Q\left( \phi ^{n+1} \right) -\frac{2}{5}q^n \rVert ^2\le \bar{E}_{2}^{n+1}.
            $$
            By the definition of $\bar{E}_{2}^{n+1}$ and above inequation, we obtain
            \begin{align*}
                {}&\frac{1}{4}\left[ \left( \mathcal{L}\phi ^{n+1},\phi ^{n+1} \right) +\left( \mathcal{L}\left( 2\phi ^{n+1}-\phi ^n \right) ,2\phi ^{n+1}-\phi ^n \right) \right]\\
                {}&+\frac{1}{2}\left( \lVert q^{n+1} \rVert ^2+\lVert 2q^{n+1}-q^n \rVert ^2 \right)\\
                {}&-\frac{1}{4}\left[ \left( \mathcal{L}\phi ^n,\phi ^n \right) +\left( \mathcal{L}\left( 2\phi ^n-\phi ^{n-1} \right) ,2\phi ^n-\phi ^{n-1} \right) \right] \\
                {}&-\frac{1}{2}\left( \lVert q^n \rVert ^2+\lVert 2q^n-q^{n-1} \rVert ^2 \right)\\
                \le{}& -\tau M \left( \mathcal{G}\mu ^{n+1},\mu ^{n+1} \right) .
            \end{align*}
            \begin{comment}
            \begin{align*}
                {}&\frac{1}{4}\left[ \left( \mathcal{L}\phi ^{n+1}\!,\phi ^{n+1} \right) \!+\!\left( \mathcal{L}( 2\phi ^{n+1}-\phi ^n ) ,2\phi ^{n+1}\!-\!\phi ^n \right) \right] \!+\!\frac{1}{2}\left( \lVert q^{n+1} \rVert ^2\!+\!\lVert 2q^{n+1}\!-\!q^n \rVert ^2 \right) \\
                {}&-\frac{1}{4}\left[ \left( \mathcal{L}\phi ^n\!,\phi ^n \right) \!+\!\left( \mathcal{L}( 2\phi ^n\!-\!\phi ^{n-1}) ,2\phi ^n\!-\!\phi ^{n-1} \right) \right]\! -\!\frac{1}{2}\left( \lVert q^n \rVert ^2\!+\!\lVert 2q^n\!-\!q^{n-1} \rVert ^2 \right) \\
                \le{}&-\tau \left( \mathcal{G}\mu ^{n+1},\mu ^{n+1} \right) .
            \end{align*}
            \end{comment}
        \end{itemize}
        The proof is completed.
    \end{proof}

\begin{remark}
    It can be analyzed by a similar proof procedure of Theorem \ref{th-e4} that
    %based on 
    for the EOP-IEQ/BDF$k$ $(k=1,2)$ schemes, %of the backward difference formula, 
     the energy-optimized technique is always equal to or superior to the relaxation technique,
     for any artificial parameter $\eta=[0,1]$.
\end{remark}

\subsection{ energy-optimized techniques for the variations of the baseline IEQ method}
Due to the wide applicability of the baseline IEQ method, 
several variants have been proposed to be suitable for different situations \cite{GONG2019224,YANG2017691}.
Since our proposed energy-optimized technique does not restrict the form of auxiliary variables, 
our EOP-IEQ approach can be easily extended to gradient flow models that contain multiple auxiliary variables.
Below we describe the EOP-IEQ method applied to a gradient flow with one phase field variable $\phi$.
We consider a more general form of the free energy 
\begin{align}
    \label{ME}
	\mathcal{E}\left( \phi \right) =\frac{1}{2}\left( \mathcal{L}\phi ,\phi \right) +\sum_{i=1}^k{\left( F_i\left( \phi \right) ,1 \right)},
\end{align}
where the operator $\mathcal{L}$ is a linear self-adjoint non-negative definite operator 
and $F_i\left( \phi \right) $, $i=1$, $2$, $\cdots$, $k$ are the bulk potentials.
The general gradient flow system of the free energy \eqref{ME} as follows
\begin{equation}
    \label{MPF}
    \begin{aligned}
        {}&\frac{\partial \phi}{\partial t}=-M\mathcal{G}\mu ,\\
        {}&\mu =\mathcal{L}\phi +\sum_{i=1}^k{F_{i}^{'}\left( \phi \right)},
    \end{aligned}
\end{equation}
which has the following energy dissipation law
$$
\frac{dE\left( \phi \right)}{dt}=-\mathcal{G}\left( \mu ,\mu \right)\le 0 ,
$$
where $M$ is a mobility constant and the operator $\mathcal{G}$ is a non-negative definite operator.
For the system \eqref{MPF}, assume that each nonlinear bulk potentials have a lower bound $F_i(\phi)\ge C^*_i$, 
$i=1$, $2$, $\cdots$, $k$. Denote $\boldsymbol{q}=\left( q_1,q_2,\dots ,q_k \right)$, 
we introduce multiple auxiliary variables 
$$
    q_i( \boldsymbol{x},t ) =:Q_i(\phi( \boldsymbol{x},t )) =\sqrt{F_i( \phi( \boldsymbol{x},t )  ) +C^0_i},\quad C^0_i>C^*_i,i=1,2,\dots,k.
$$
Then the equivalent form of the gradient flow system \eqref{MPF} can be written as
\begin{equation}
    %\begin{subequations}
    \label{MIEQ}
    \begin{aligned}
	{}&\frac{\partial \phi}{\partial t}=-M\mathcal{G}\mu ,\\
	{}&\mu =\mathcal{L}\phi +\sum_{i=1}^k{q_iH_i( \phi )},,\\
	{}&\frac{\partial q_j}{\partial t}=\frac{H_j( \phi )}{2}\frac{\partial \phi}{\partial t},\quad j=1,2,\dots ,k.
	\end{aligned}
    %\end{subequations}
\end{equation}
where $H_i( \phi ) =\frac{f_i( \phi )}{Q_i( \phi )},f_i=F_{i}^{'}\left( \phi \right),i=1,2,\dots ,k.$
Then we can get the modified free energy as
\begin{align}
	\bar{\mathcal{E}}\left( \phi,\boldsymbol{q} \right) =\frac{1}{2}\left( \mathcal{L}\phi ,\phi \right) 
    +\sum_{i=1}^k{\lVert q_i \rVert ^2}-\sum_{i=1}^k{C^0_i|\Omega|}.
\end{align}
For the reformulated system \eqref{MIEQ}, it has the following modified energy dissipation law
$$
\frac{d\bar{E}( \phi ,\boldsymbol{q} )}{dt}=-M\mathcal{G}( \mu ,\mu ) \le 0.
$$
We introduce the energy-optimized technique into the system \eqref{MIEQ} to improve the inconsistency between 
the modified energy and the original energy after discretization. 
%Then a semi-implicit second-order IEQ scheme based on Crank-Nicolson is as follows
Then a second-order EOP-IEQ scheme based on Crank-Nicolson formula can be written as follows.
\begin{corollary} \label{EOP-MIEQ}
    For initial datas, we set $\phi ^0=\phi \left( \boldsymbol{x},0 \right)$, $q^0=\sqrt{F\left( \phi ^0 \right) +C_0}$,
    then we update $\phi^{n+1}$, $q^{n+1}$ via the following two steps:
    \begin{itemize}
        \item[] $\bf{Step~1} : $
        To obtain the numerical solution $\phi^{n+1}$ by the baseline IEQ method: 
        %\begin{subequations}
            \begin{align*}
                {}&\frac{\phi ^{n+1}-\phi ^n}{\tau}=-M\mathcal{G}\mu ^{n+\frac{1}{2}},\\
                {}&\mu ^{n+\frac{1}{2}}=\mathcal{L}\frac{\phi ^{n+1}+\phi ^n}{2}+\sum_{i=1}^k{\frac{\hat{q}_{i}^{n+1}+q_{i}^{n}}{2}H_i( \bar{\phi}^{n+\frac{1}{2}} )},\\
                {}&\frac{\hat{q}_{j}^{n+1}-q_{j}^{n}}{\tau}=H_j( \bar{\phi}^{n+\frac{1}{2}}) \frac{\phi ^{n+1}-\phi ^n}{\tau},\quad j=1,2,\dots ,k.
            \end{align*}
        %\end{subequations}
        where
        $$
         H_i( \bar{\phi}^{n+\frac{1}{2}} ) =\frac{f_i( \bar{\phi}^{n+\frac{1}{2}} )}{Q_i( \bar{\phi}^{n+\frac{1}{2}} )},\quad
         \bar{\phi}^{n+\frac{1}{2}}=\frac{3}{2}\phi ^n-\frac{1}{2}\phi ^{n-1}.
        $$
        \item[] $\bf{Step~2} : $
        To update the auxiliary variables $\boldsymbol{q}^{n+1}=\left( q_{1}^{n+1},q_{2}^{n+1},\cdots ,q_{k}^{n+1} \right)$ with the following energy-optimized step:
        $$
        q_{i}^{n+1}=\lambda ^{n+1}Q_i\left( \phi ^{n+1} \right) ,\quad \lambda ^{n+1}=\min \left\{ 1,\sqrt{\frac{E^{n+1}_2}{E^{n+1}_1}} \right\} , i=1,2,\cdots ,k,
        $$
        where 
        $$
        E^{n+1}_1=\sum_{i=1}^k{\lVert Q_i\left( \phi ^{n+1} \right) \rVert ^2},\quad E^{n+1}_2=\bar{E}^n-\frac{1}{2}\left( \mathcal{L}\phi ^{n+1},\phi ^{n+1} \right) .
        $$
    \end{itemize}
\end{corollary}
\begin{theorem}\label{th-e5}
    The scheme \ref{EOP-MIEQ} is unconditionally energy stable in the sense that 
\begin{align*}
    {}&\frac{1}{2}\left( \mathcal{L}\phi ^{n+1},\phi ^{n+1} \right) +\sum_{i=1}^k{\lVert q_{i}^{n+1} \rVert ^2}-\frac{1}{2}\left( \mathcal{L}\phi ^n,\phi ^n \right) -\sum_{i=1}^k{\lVert q_{i}^{n} \rVert ^2}\\
    \le{}& -\tau M\left( \mathcal{G}\mu ^{n+\frac{1}{2}},\mu ^{n+\frac{1}{2}} \right) 
    \le 0.
\end{align*}
\end{theorem}
The proof of this theorem is not particularly difficult but will not be reproduced here.

\section{Numerical implementation and results}
In this section, we adopt some widely used gradient flow models including Allen-Cahn (AC) equation, 
Cahn-Hilliard (CH) equation, molecular beam epitaxy (MBE) model, and phase field crystal (PFC) model 
to verify the generality of EOP-IEQ method. %energy-optimized technique. 
Furthermore, the corresponding numerical algorithms are implemented to verify the above theoretical results.
%For simplicity, we consider periodic boundary conditions and use Fourier spectrum methods in spatial variables in all examples, 
For simplicity, we consider periodic boundary condition and use the Fourier spectrum method in space in all examples,
where $N_x$ and $N_y$ represent the number of Fourier modes along the $x$ and $y$ axis respectively.
It is worth noting that the proposed schemes can also be applied to other thermodynamically 
consistent boundary conditions which satisfy energy dissipation law.
%In addition, EOP-IEQ/CN scheme and EOP-IEQ/BDF$k$ are abbreviated as EOP/CN and EOP/BDF$k$ respectively.

\subsection{ Allen-Cahn equation }
To begin with, we consider the free energy \cite{newSAV}
\begin{align}
    \label{E-ACH}
    E\left( \phi \right) =\int_{\Omega}{\frac{\alpha_0}{2}\left| \nabla \phi \right|^2}+\frac{\alpha_0}{\varepsilon ^2}F\left( \phi \right) \text{d}\boldsymbol{x,}\quad F\left( \phi \right) =\frac{\left( \phi ^2-1 \right) ^2}{4},
\end{align}
where the interface parameters $0<$ $\alpha_0$, $\varepsilon$ $\ll 1$.\par 
The Allen-Cahn (AC) equation \cite{Allen1979AMT} can be expressed in the energy variational form \eqref{PF} of 
the free energy \eqref{E-ACH} with $\mathcal{G} = MI$ as follows
\begin{align}
    \label{AC}
    \frac{\partial \phi}{\partial t}=-M\left( -\alpha_0 \Delta \phi +\frac{\alpha_0}{\varepsilon ^2}f\left( \phi \right) \right).
\end{align}
where $f\left( \phi \right) =F'(\phi)=\phi ^3-\phi $.
Now we introduce the auxiliary variable
$q:=Q\left( \phi \right) =\sqrt{F\left( \phi \right) +C_0}$, then the general model \eqref{IEQ} is specified as
\begin{equation}
    \label{AC-IEQ}
    \begin{aligned}
        {}&\frac{\partial \phi}{\partial t}=-M\left( -\alpha_0 \Delta \phi +\frac{\alpha_0 }{\varepsilon ^2}\frac{q}{Q\left( \phi \right)} f\left( \phi \right)\right) ,\\
        {}&\frac{\partial q}{\partial t}=\frac{f\left( \phi \right)}{2Q\left( \phi \right)}\frac{\partial \phi}{\partial t}.
    \end{aligned}
\end{equation}

\textit{Case A.}
We check the convergence rate in time of the proposed schemes for the AC equation under the $2D$ computational region $\Omega = [-0.5,0.5]^2$.
The initial data is chosen to be smooth
\begin{align}
    \label{AC-init}
    \phi \left( x,y,0 \right) =\tanh\frac{a+b\cos \left( 6\theta \right) -c \sqrt{x^2+y^2}}{\sqrt{2}\varepsilon},
    \quad \theta =\arctan \frac{y}{x},
\end{align}
where $a,b,c$ are positive constants.
We set $(N_x, N_y) =(64,64)$, so that the spatial discrete error is negligible compared to the temporal discrete error.
The parameters are $(a,b,c) = (1.5,1.2,2\pi)$, $\alpha_0=0.01$, $\varepsilon =0.01 $, $M=0.6$, $C_0=100$. % and $T=0.1$.
In addition, since the exact solution is unknown, we choose the EOP-IEQ/CN scheme with time step $\tau = 1e-6$ 
as the reference solution for the calculation error.
In Fig. \ref{F11}, we obtain the numerical errors of $L^2$ norm and convergence rates at 
$T = 0.1$ using various first- and second-order schemes. We observe that 
(i) the temporal convergence rates obtained by baseline IEQ and EOP-IEQ methods 
are consistent with theoretical expectations in all cases; 
(ii) the $L^2$ errors of the numerical solutions $\phi$ 
and the auxiliary variable $q$ in 
the EOP-IEQ schemes are significantly smaller than those in the baseline IEQ schemes.

\begin{figure}[t]
    \centering
    \subfloat[BDF1]{\includegraphics[width=0.33\textwidth]{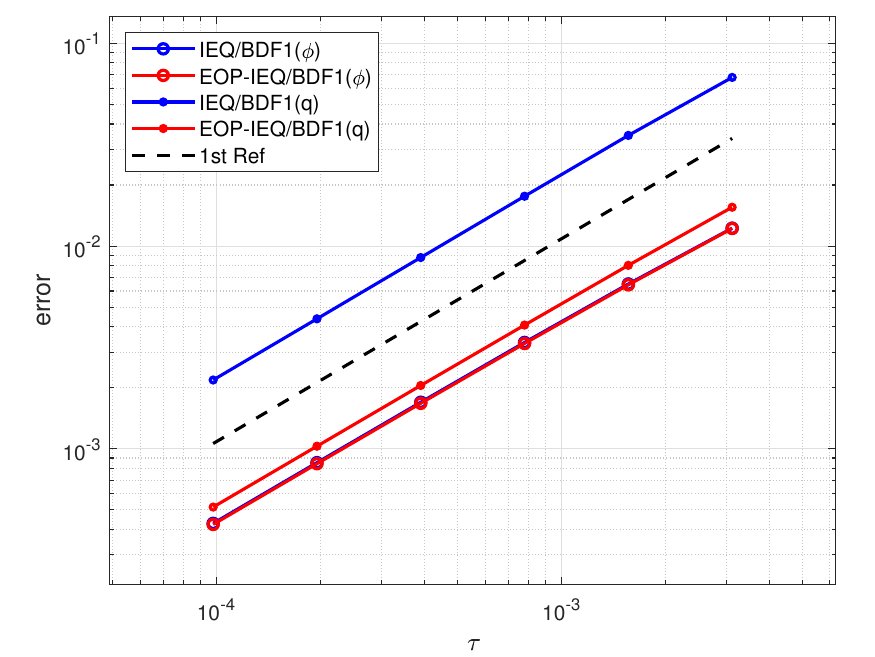}}
    \hfill
    \subfloat[BDF2]{\includegraphics[width=0.33\textwidth]{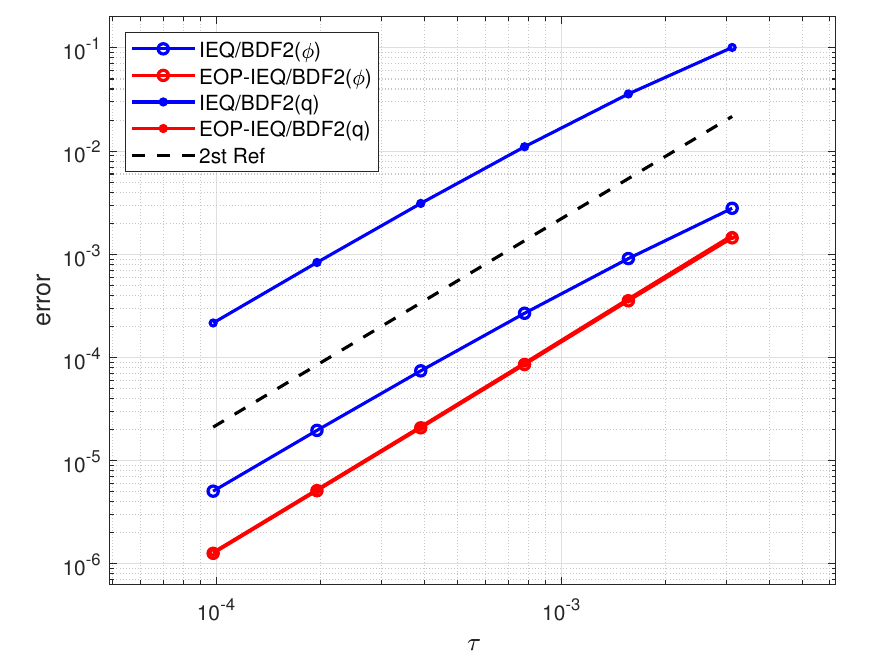}}
    \hfill
    \subfloat[CN]{\includegraphics[width=0.33\textwidth]{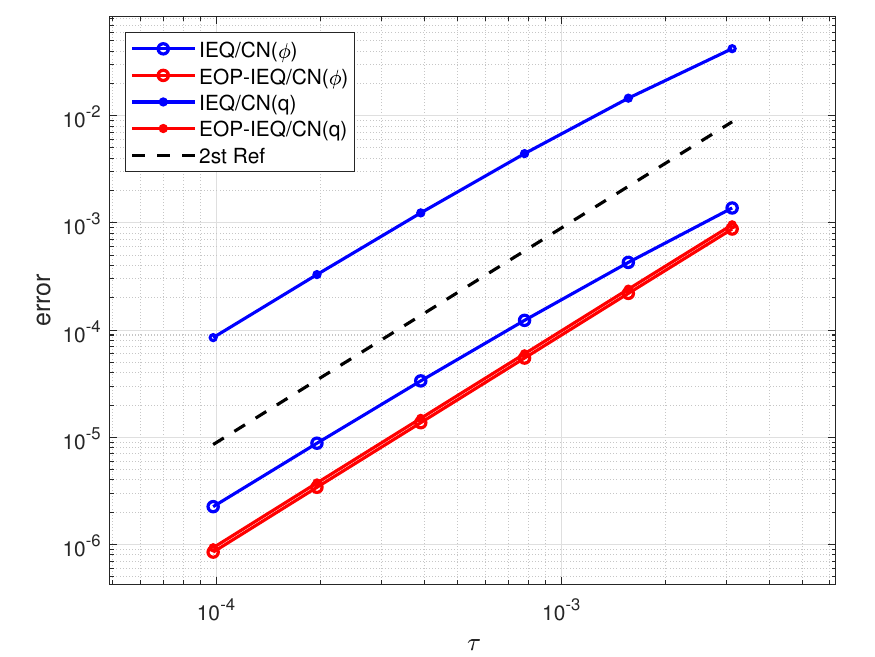}}
    \hfill
    \caption{Example 1A. Errors and convergence rates in the $L^2$ norm of the numerical solution $\phi$ 
    and the auxiliary variable $q$ for AC equation by using various first- and second-order schemes.}
    %$ (a):BDF1; (b):BDF2; (c):CN $. }
    \label{F11}
\end{figure}

\textit{Case B.}
We also choose the initial condition \eqref{AC-init}
to compare the accuracy of the baseline IEQ, REQ and the EOP-IEQ methods for solving the AC equation.
We adjust some parameters $(N_x, N_y) =(256,256)$, $\eta=0.5$, $M=0.51$, $C_0=1$, $\tau=1e-2$ and $T=10.5$
and use the results of the EOP-IEQ/BDF2 scheme of $\tau=1e-6$ as the reference solution.
The profiles of $\phi$ at various times are summarized in Fig. \ref{F13}. 
A comparison of the modified energy (left), the energy error (middle), 
and the update parameters (right) for the IEQ/CN, REQ/CN, and EOP-IEQ/CN schemes shown in Fig. \ref{F14},
where the update parameters include relaxation parameters $\xi^n$ 
and energy-optimized parameters $\lambda^n$, where $\lambda_1^n=|\lambda^n-1|$.
We observe that the energy-optimized parameters $\lambda_1^n=0, \forall n\ge 1$ in the EOP-IEQ/CN scheme,
which means $\lambda^n=1$ and $q^{n}=Q(\phi^n), \forall n\ge 1$.
It is essential to preserve the consistency of the modified energy with the original energy.
%Meanwhile, we know that the numerical solutions of $q^{n}$ and $Q(\phi^n)$ of REQ/CN scheme is not always equal in Fig. \ref{F14} (c).
Meanwhile, Fig. \ref{F14} (b) shows that 
the errors between the corrected energy and the reference energy using the IEQ/CN and REQ/CN schemes are significantly larger.
It is verified that our proposed EOP-IEQ approach ensure the modified energy closer to original energy
than the REQ approach.
\begin{comment}
It indicates the energy-optimized technique improves the accuracy significantly. 
%the EOP-IEQ/CN scheme can provide accurate numerical results.
In Fig. \ref{F14}, we show a comparison of the modified energy (left), the energy error (middle), 
and the nonlinear free energy ratio %i.e. $\frac{\lVert q^n \rVert ^2}{\lVert Q\left( \phi ^n \right) \rVert ^2}$
(right) for the IEQ/CN scheme, REQ/CN scheme, and EOP-IEQ/CN scheme.
It can be observed that the ratio of nonlinear free energy $\frac{1}{4}\lVert Q\left( \phi ^n \right) \rVert ^2$ and $\frac{1}{4}\lVert q^n \rVert ^2$
in the EOP-IEQ/CN scheme is always $1$, 
indicating that 
%$q^{n}=Q(\phi^n)$ at each time step,and 
it is essential to preserve the consistency of the modified energy with the original energy.
Meanwhile, we know taht the numerical solutions of $q^{n}$ and $Q(\phi^n)$ of REQ/CN scheme is not always equal in Fig. \ref{F14} (c).
It is verified that our proposed EOP-IEQ approach ensure the modified energy closer to original energy
than the REQ approach.
\end{comment}
\begin{figure}[htp]
    \centering
    \subfloat[$t=0.1$]{\includegraphics[width=0.33\textwidth]{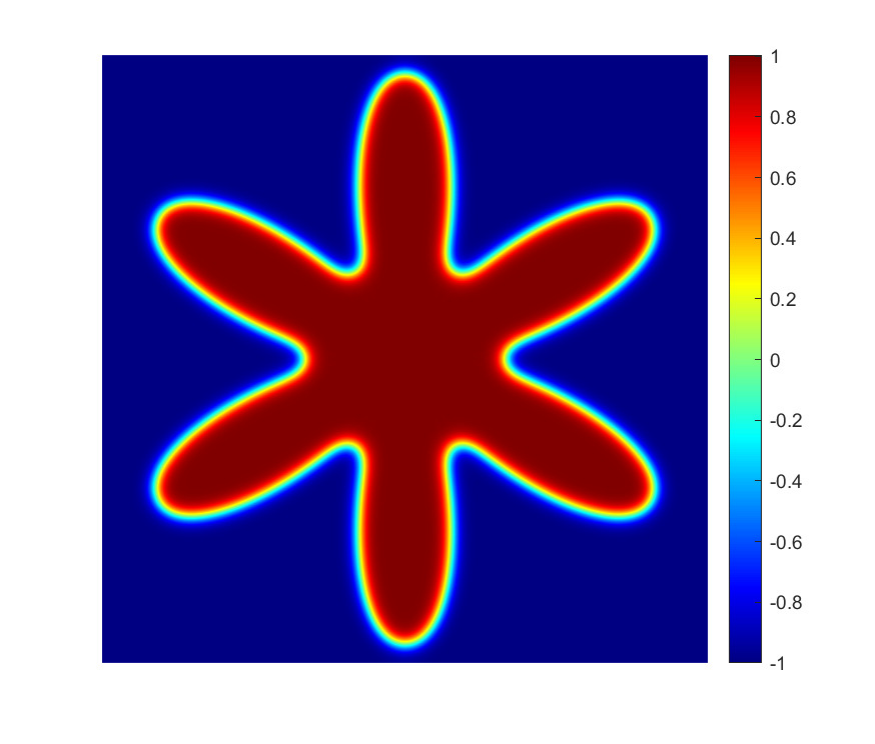}}
    \hfill
    \subfloat[$t=1.5$]{\includegraphics[width=0.33\textwidth]{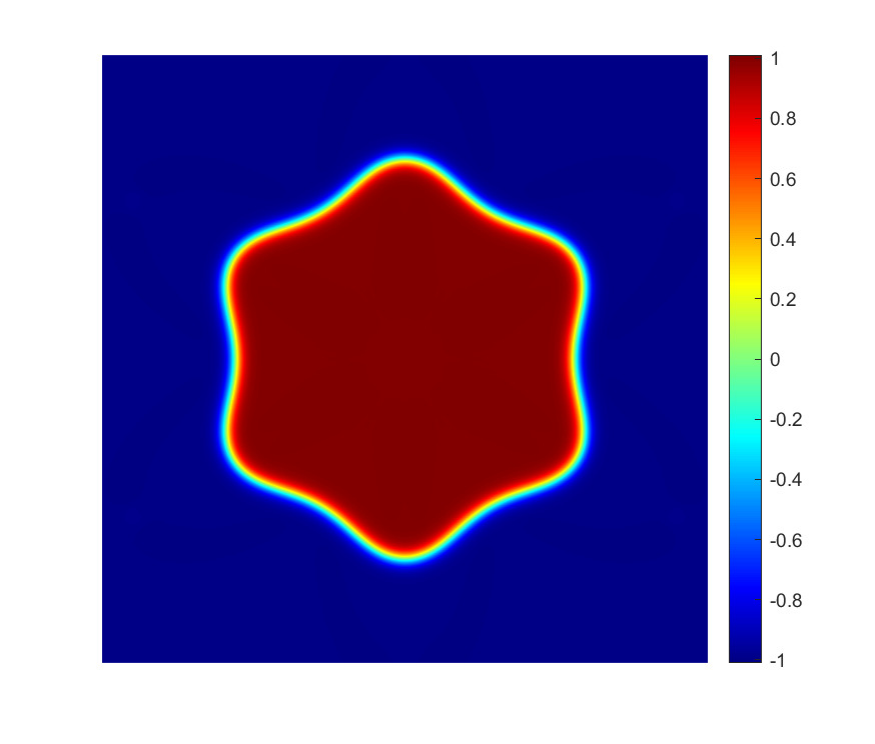}}
    \hfill
    \subfloat[$t=9$]{\includegraphics[width=0.33\textwidth]{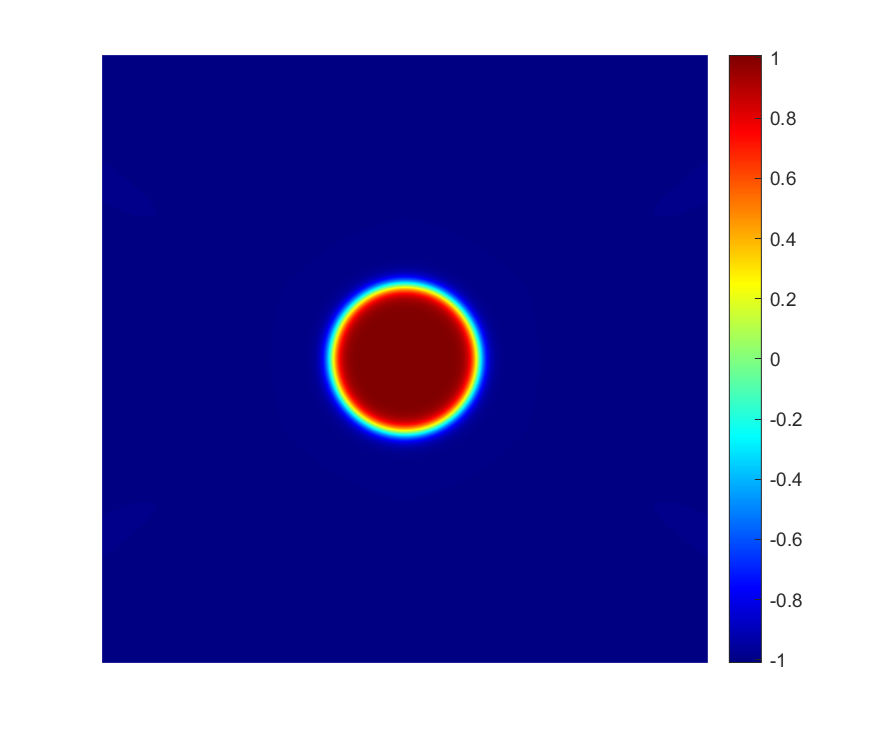}}%Liu3
    \hfill
    \caption{Example 2A. Profiles of the phase variable $\phi$ 
    are taken at $t=0.1, 1.5, 9$.}
    \label{F13}
\end{figure}
\begin{comment}
    \begin{figure}[htp]
        \centering
        \subfloat[]{\includegraphics[width=0.33\textwidth]{ACE.eps}}
        \hfill
        \subfloat[]{\includegraphics[width=0.33\textwidth]{ACErr.eps}}
        \hfill
        \subfloat[]{\includegraphics[width=0.33\textwidth]{ACXi.eps}}
        \hfill
        \caption{Example 2A. 
        A comparison between the baseline IEQ/CN scheme, REQ/CN scheme and EOP-IEQ/CN scheme for solving the AC equation. 
        %(a) Evolution of normalized energy under different numerical schemes;
        %(b) Evolution of relative mass deviation $m(t)/m(0)$ with various numerical schemes.
        (a) Normalized numerical energy comparisons between the three methods;
        (b) Evolution of the energy error between the three methods;
        (c) Time evolution of nonlinear free energy ratio $\frac{1}{4}\lVert Q\left( \phi ^n \right) \rVert ^2$ and $\frac{1}{4}\lVert q^n \rVert ^2$ of REQ/CN scheme and EOP-IEQ/CN scheme.
        }
        \label{F14}
    \end{figure}
\end{comment}
\begin{figure}[htp]
    \centering
    \subfloat[]{\includegraphics[width=0.33\textwidth]{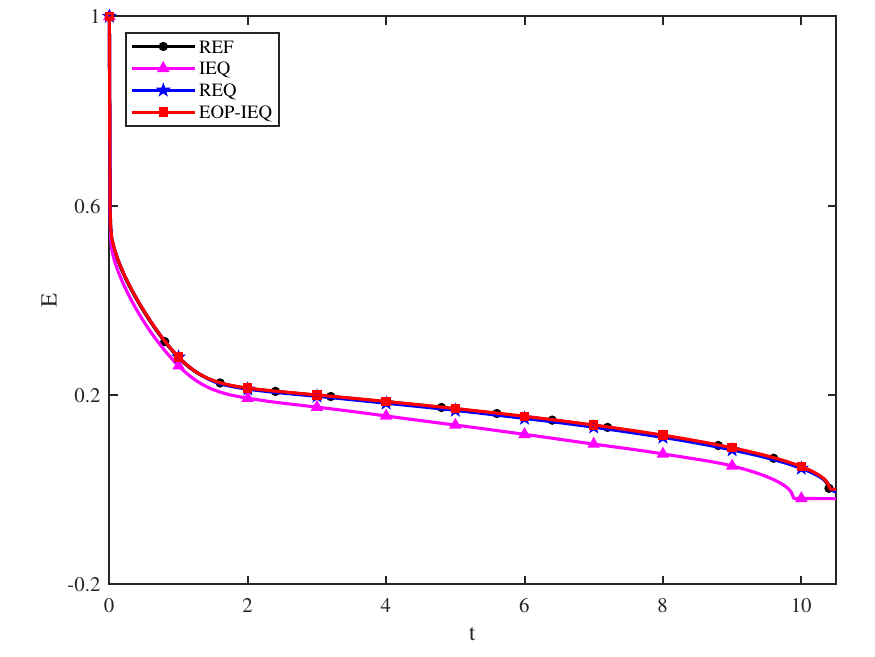}}
    \hfill
    \subfloat[]{\includegraphics[width=0.33\textwidth]{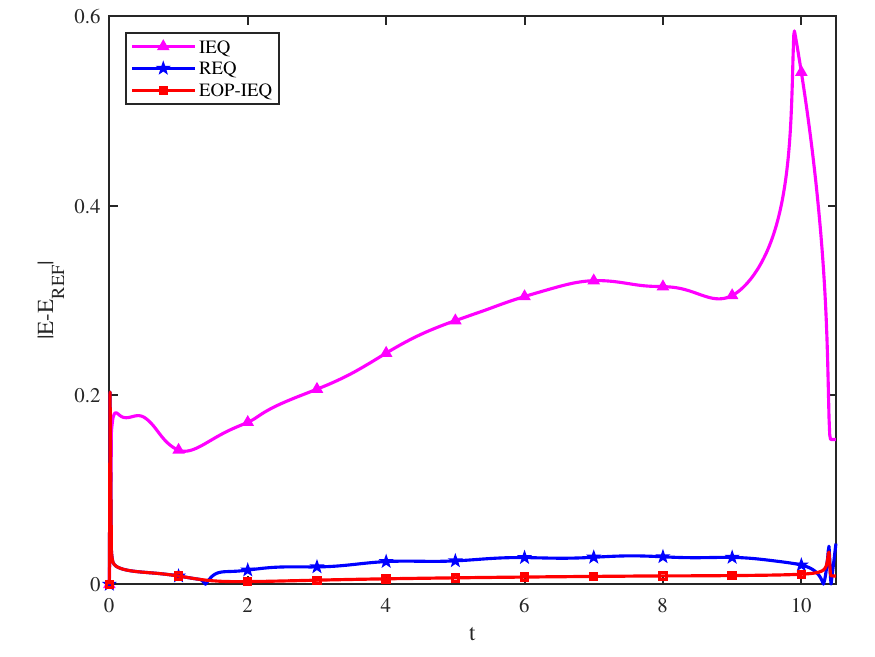}}
    \hfill
    \subfloat[]{\includegraphics[width=0.33\textwidth]{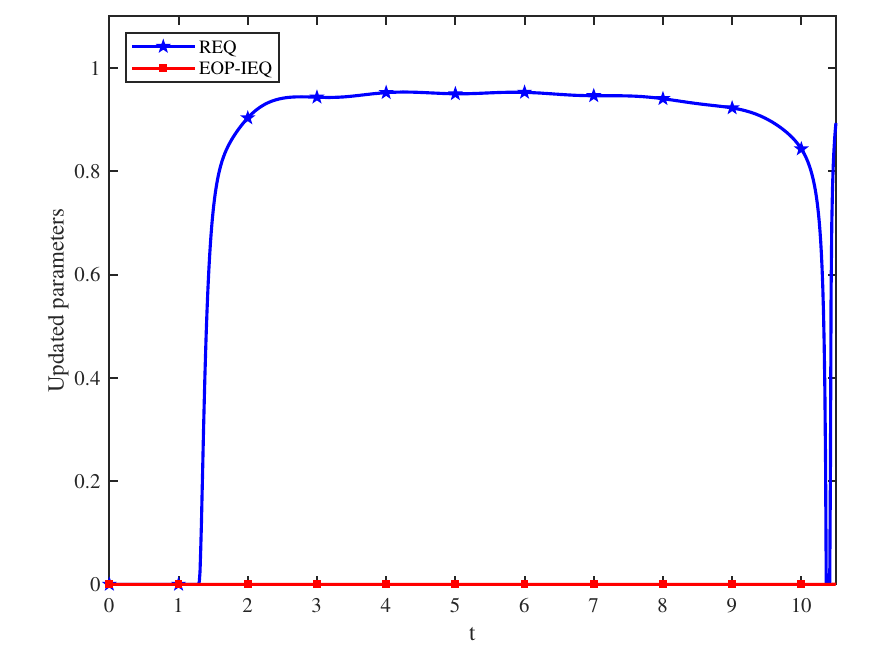}}
    \hfill
    \caption{Example 2A. 
    A comparison between the baseline IEQ/CN scheme, REQ/CN scheme and EOP-IEQ/CN scheme for solving the AC equation. 
    %(a) Evolution of normalized energy under different numerical schemes;
    %(b) Evolution of relative mass deviation $m(t)/m(0)$ with various numerical schemes.
    (a) Normalized numerical energy comparisons of the three schemes;
    (b) Evolution of the the errors between modified energies and reference energy of the three schemes;
    (c) Time evolution of energy-optimized parameters $\lambda_1^n$ in EOP-IEQ/CN scheme 
        and relaxation parameters $\xi^{n}$ in REQ/CN scheme. 
    }
    \label{F14}
\end{figure}

\subsection{ Cahn-Hilliard equation }
Now we consider the Cahn-Hilliard (CH) equation given as
\begin{equation}
    \label{CH}
    \begin{aligned}
        {}&\frac{\partial \phi}{\partial t}=-M\Delta \mu ,\\
        {}&\mu =-\alpha _0\Delta \phi +\frac{\alpha _0}{\varepsilon ^2}f\left( \phi \right) .
    \end{aligned}
\end{equation}
%where $f\left( \phi \right) =\phi ^3-\phi $.
It can be expressed in the energy variational form \eqref{PF} of the free energy \eqref{E-ACH} with $\mathcal{G} = -M\Delta$.
We introduce the auxiliary variable $q:=Q\left( \phi \right) =\sqrt{F\left( \phi \right) -\kappa\phi ^2/2+C_0}$, 
then the general model in \eqref{IEQ} is specified as
\begin{equation}
    \begin{aligned}
        {}&\frac{\partial \phi}{\partial t}=-M\Delta \left( -\alpha _0\Delta \phi +\frac{\alpha _0}{\varepsilon ^2}\kappa \phi +\frac{q}{Q\left( \phi \right)}\frac{\alpha _0}{\varepsilon ^2}f_{\kappa}\left( \phi \right) \right) ,\\
        {}&\frac{\partial q}{\partial t}=\frac{f_{\kappa}\left( \phi \right)}{2Q\left( \phi \right)}\frac{\partial \phi}{\partial t},
    \end{aligned}
\end{equation}
where $f_{\kappa}\left( \phi \right) =\phi^3-\phi-\kappa\phi $.
In addition, the CH equation satisfies the law of conservation of mass
\begin{align}
    \label{MM}
    \frac{dm\left( t \right)}{dt}=0,\quad m\left( t \right) =\int_{\Omega}{\phi \left( \boldsymbol{x,}t \right)}\,\text{d}\boldsymbol{x}.
\end{align}

\textit{Case A.}
The initial conditions \cite{Cheng2021GeneralizedSA} is considered as the following discrete form 
\begin{align}
    \label{CH_init}
    \phi \left( x,y,0 \right) =0.25+0.4rand\left( x,y \right) ,
\end{align}
where $rand(x,y)$ represents a pseudo-random value derived from a uniform distribution within $(-1,1)^2$.
%Other parameters are
We choose the numerical parameters
$\Omega = [-0.5,0.5]^2, (N_x,N_y)=(64,64)$, $\alpha_0=0.01^2$, $\varepsilon = 0.01$, $\kappa=4$, $M=0.001$, $C_0=100$, 
$T=0.1$, and the reference solution is computed by the EOP-IEQ/BDF2 scheme for $\tau=1e-6$.
The numerical errors and convergence rates of EOP-IEQ method are shown in Fig \ref{F21}.
We observed that the numerical results of the numerical solution $\phi$ 
and the auxiliary variable $q$ in $L^2$ errors obtained by using the proposed schemes are in agreement with expectations.

\begin{figure}[htp]
    \centering
    \subfloat[BDF1]{\includegraphics[width=0.33\textwidth]{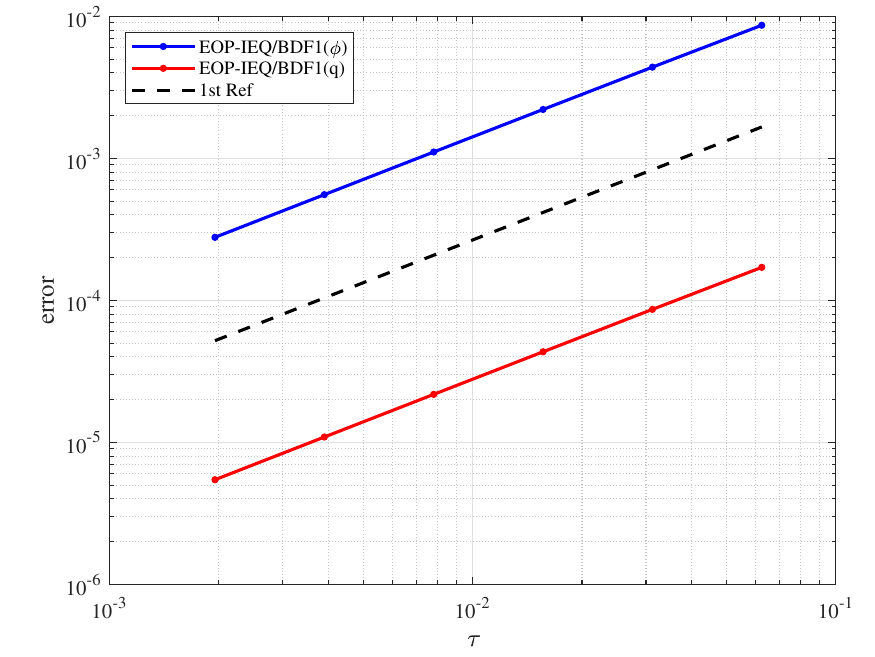}}
    \hfill
    \subfloat[BDF2]{\includegraphics[width=0.33\textwidth]{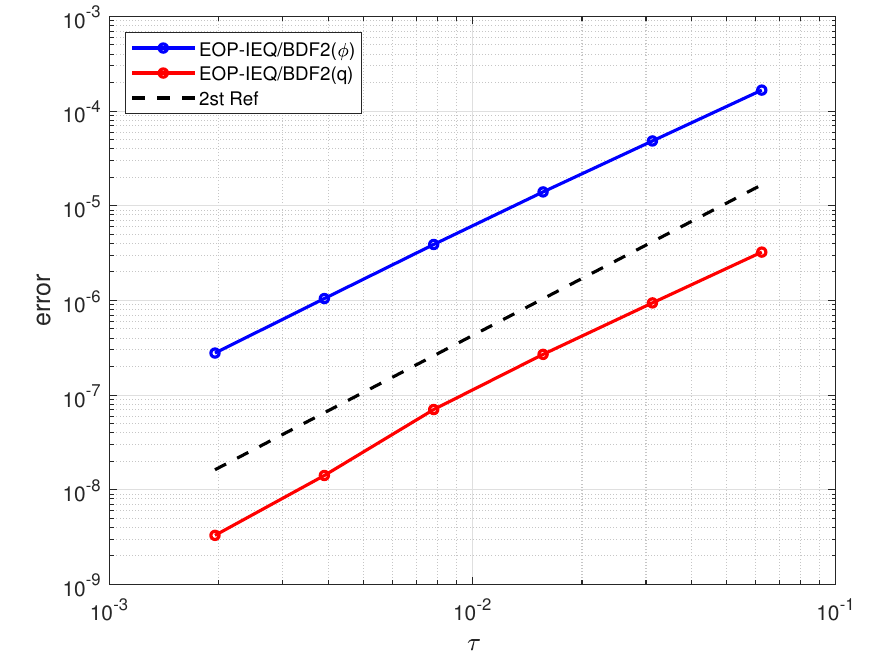}}
    \hfill
    \subfloat[CN]{\includegraphics[width=0.33\textwidth]{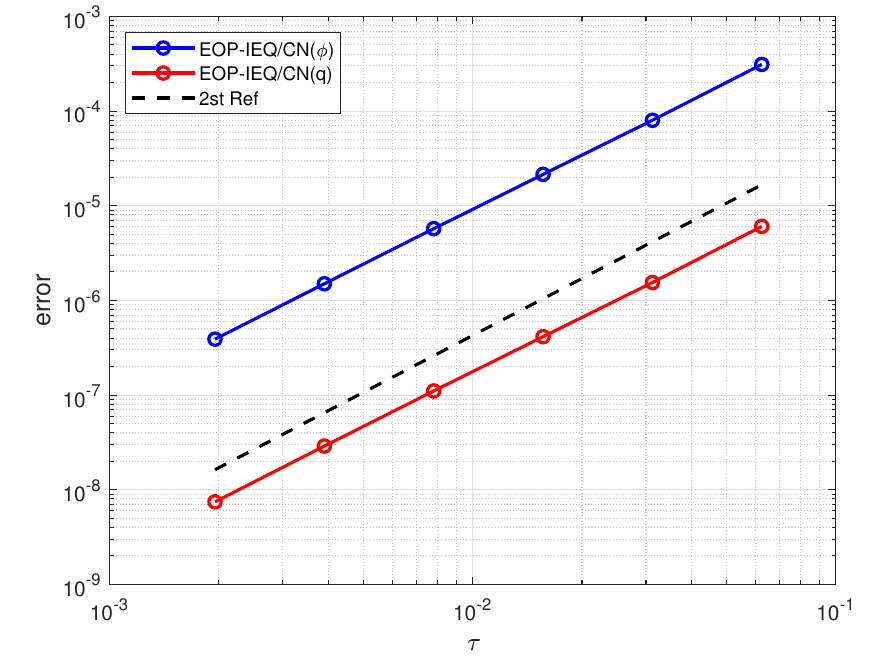}}
    \hfill
    \caption{Example 2A. Errors and convergence rates in the $L^2$ norm of the numerical solution $\phi$ 
    and auxiliary variable $q$ for CH equation using various first- and second-order schemes.}
    \label{F21}
\end{figure}

\textit{Case B.}
Considering that coarsening effect is an important phenomenon in phase field simulation, 
this example is used to verify the validity of using energy-optimized technique to simulates long-time dynamics.
The initial data is set to be the same as \eqref{CH_init}
and we adjust some parameters to $(N_x,N_y)=(128,128)$, $\tau=1e-3$, $M=0.1$ and $T=100$.
The results of the EOP-IEQ/CN scheme of $\tau=1e-5$ are calculated as the reference solution.
Fig. \ref{F23} shows several profiles of coarsening process under the EOP-IEQ/CN scheme 
at $t=1, 10, 90$.
We use the IEQ/CN and EOP-IEQ/CN schemes in Fig. \ref{F24} to 
investigate the modified energy, the ratio of nonlinear free energy and mass profiles during the evolution process.
We observe that the modified energy obtained by the EOP-IEQ method 
is closer to the original energy than that obtained by the baseline IEQ method.
It indicates the energy-optimized technique improves the accuracy significantly.

\begin{figure}[htp]
    \centering
    \subfloat[$t=1$]{\includegraphics[width=0.33\textwidth]{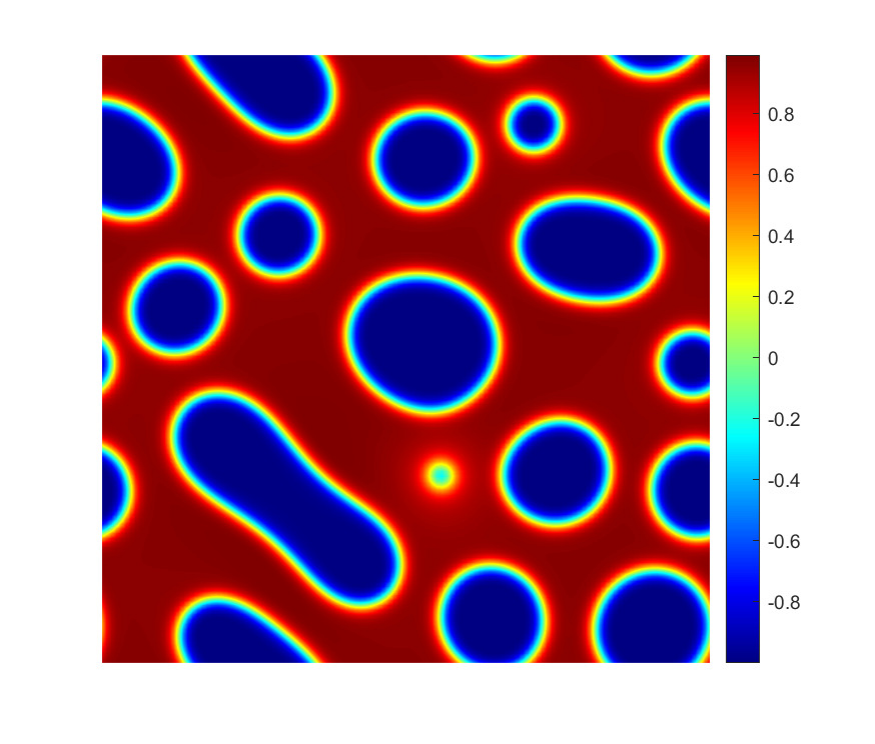}}
    \hfill
    \subfloat[$t=10$]{\includegraphics[width=0.33\textwidth]{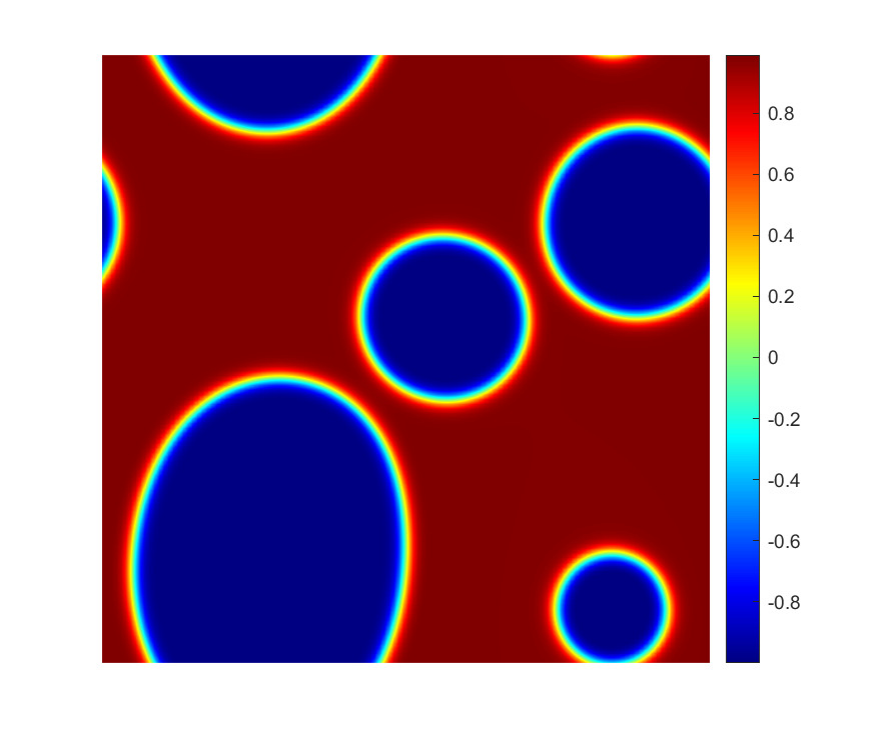}}
    \hfill
    \subfloat[$t=90$]{\includegraphics[width=0.33\textwidth]{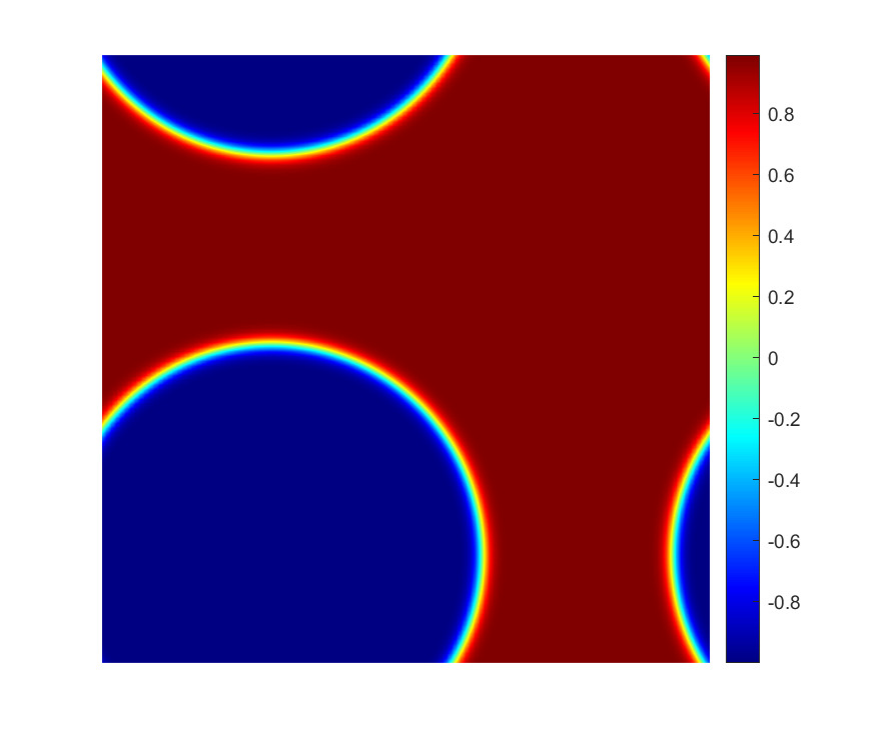}}
    \hfill
    \caption{Example 2A. Profiles of the phase variable $\phi$ 
    are taken at $t=1, 10, 90$.}
    \label{F23}
\end{figure}
\begin{figure}[htp]
    \centering
    \subfloat[]{\includegraphics[width=0.33\textwidth]{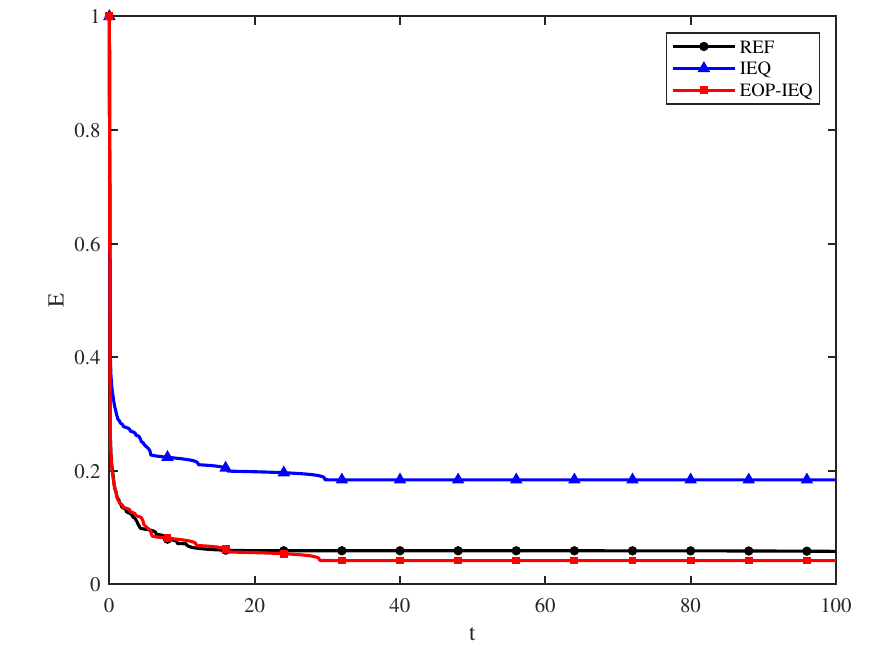}}
    \hfill
    \subfloat[]{\includegraphics[width=0.33\textwidth]{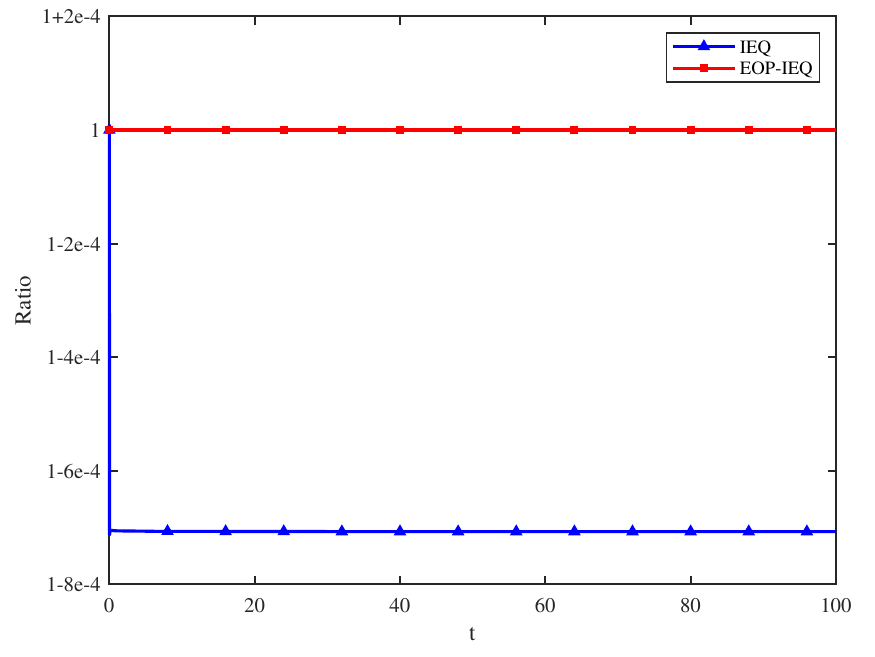}}
    \hfill
    \subfloat[]{\includegraphics[width=0.33\textwidth]{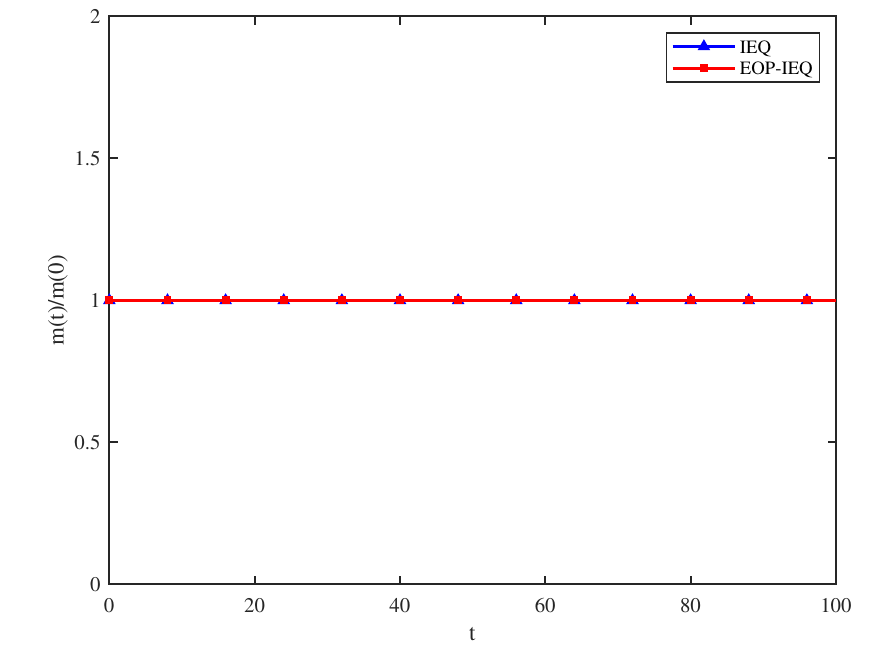}}
    \hfill
    \caption{Example 2A. 
    A comparison between the baseline IEQ/CN scheme and the EOP-IEQ/CN scheme for solving the CH equation. 
    %(a) Evolution of normalized energy under different numerical schemes;
    %(b) Evolution of relative mass deviation $m(t)/m(0)$ with various numerical schemes.
    (a) Normalized numerical energy comparisons of the two schemes;
    (b) Evolution of the ratio of nonlinear free energy 
        $\lVert Q\left( \phi ^n \right) \rVert ^2/4$ and $\lVert q^n \rVert ^2/4$;
    (c) Time evolution of relative mass deviation $m(t)/m(0)$.
    }
    \label{F24}
\end{figure}

\textit{Case C.}
In the following examples, we will use the second-order EOP-IEQ/CN scheme \eqref{EOP/CN} for numerical simulation if not specified.
We study the evolution of isolated ellipses and circles \cite{Huang2024ACO}. 
The initial condition is  
$$
\phi \left( x,y,0 \right) =\tanh\left( \frac{dist\left( \left( x,y \right) ,\varGamma \right)}{\sqrt{2}\varepsilon} \right) ,
$$
where $dist\left( \left( x,y \right) ,\varGamma \right)$ represents the signed distance from point $\left( x,y \right)$ 
to the ellipse and circular interfaces $\varGamma$. The center of the ellipse is located at 
$(-0.1, -0.1 )$, its major and minor axes are $\sqrt{2}/5$, $\sqrt{2}/10$,
the central position of the circle is $(0.25,0.25)$, and the radius is set to $0.1$.
Other parameters are $(N_x,N_y)=(128,128)$, $\alpha_0=0.01^2$, $\tau=1e-4$, $\varepsilon = 0.01$, $\kappa=4$, $M=1$, $C_0=100$
and $T=1$. We use the results of the EOP-IEQ/CN scheme of $\tau=1e-6$ as the reference solution.
Fig. \ref{F25} shows the morphological evolution under the IEQ/CN scheme with or without energy-optimized technique, 
and it can also be seen that the small circle is gradually absorbed by the ellipse until it completely disappears. 
We observe that the numerical results of the EOP-IEQ/CN scheme are closer to the reference solution from the enlarged figures. 
At the same time, it can be observed that the modified energy using EOP-IEQ/CN scheme 
is closer to the reference energy from the energy curve in Fig. \ref{F26} (a), (b), 
while the baseline IEQ method has significant deviation. 
In addition, we observe that the baseline IEQ method can maintain the conservation of 
mass regardless of the energy-optimized technique in Fig. \ref{F26} (c).

\begin{figure}[htp]
    \centering
    \subfloat[$t=0.02$]{\includegraphics[width=0.33\textwidth]{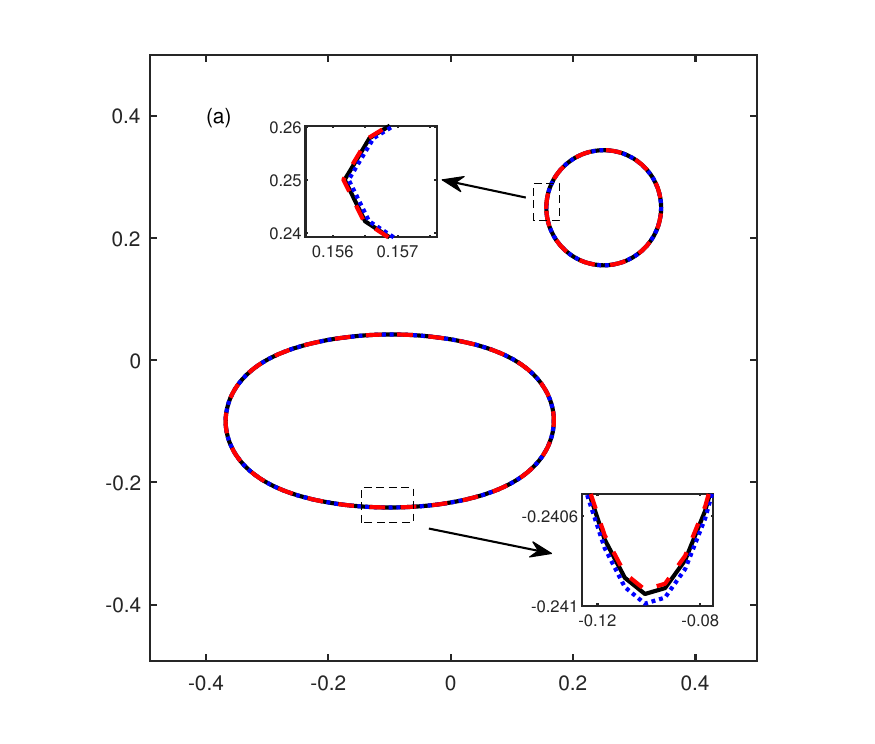}}
    \hfill
    \subfloat[$t=0.3$]{\includegraphics[width=0.33\textwidth]{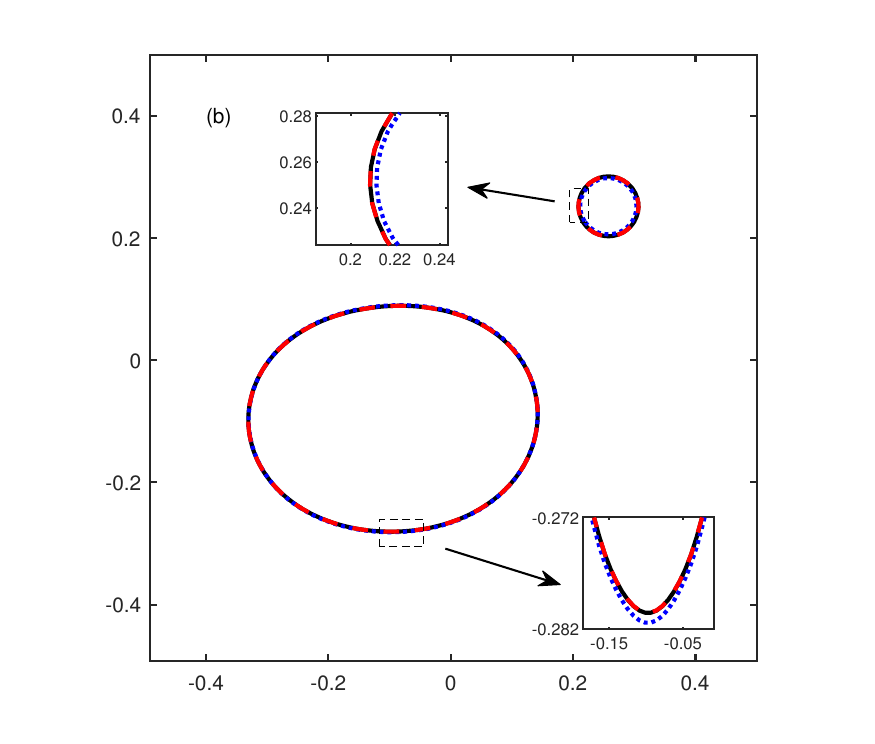}}
    \hfill
    \subfloat[$t=1$]{\includegraphics[width=0.33\textwidth]{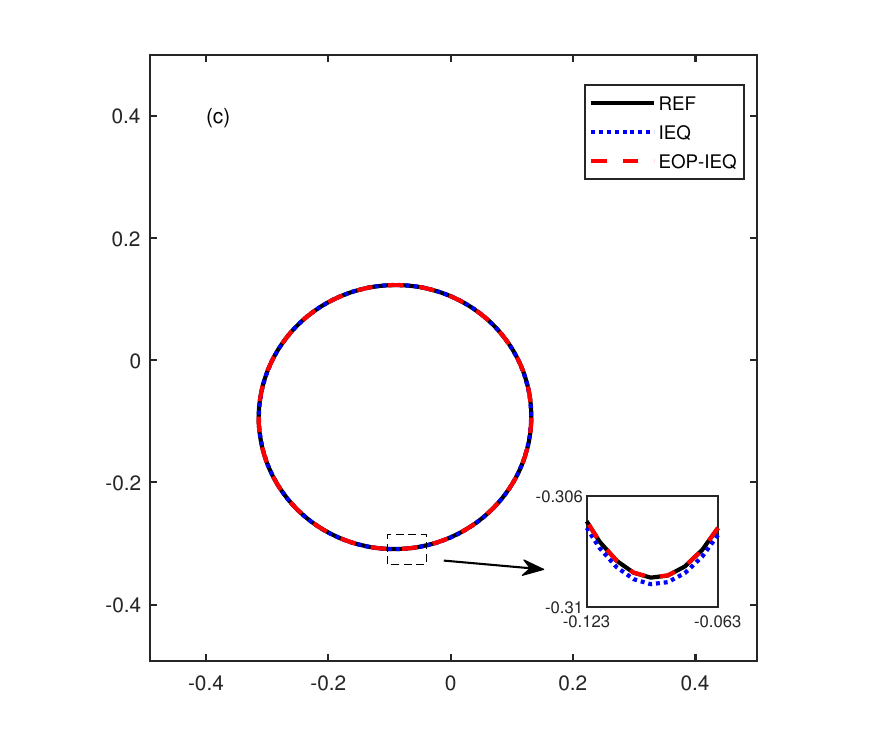}}
    \hfill
    \caption{Example 2B. Profiles of the phase variable $\phi$ 
    are taken at $t=0.02,0.3,1$.}
    \label{F25}
\end{figure}
\begin{figure}[htp]
    \centering
    \subfloat[]{\includegraphics[width=0.33\textwidth]{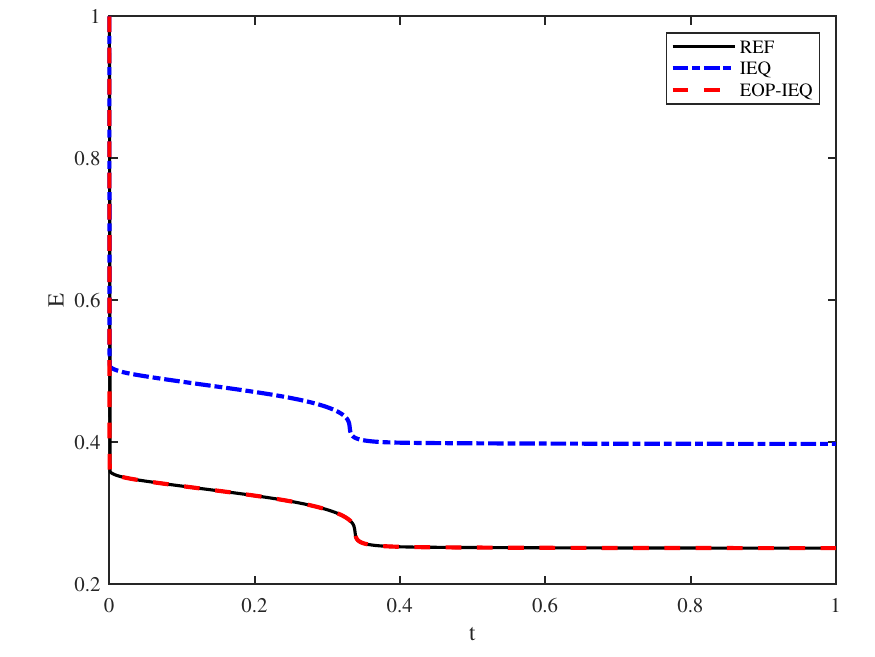}}
    \hfill
    \subfloat[]{\includegraphics[width=0.33\textwidth]{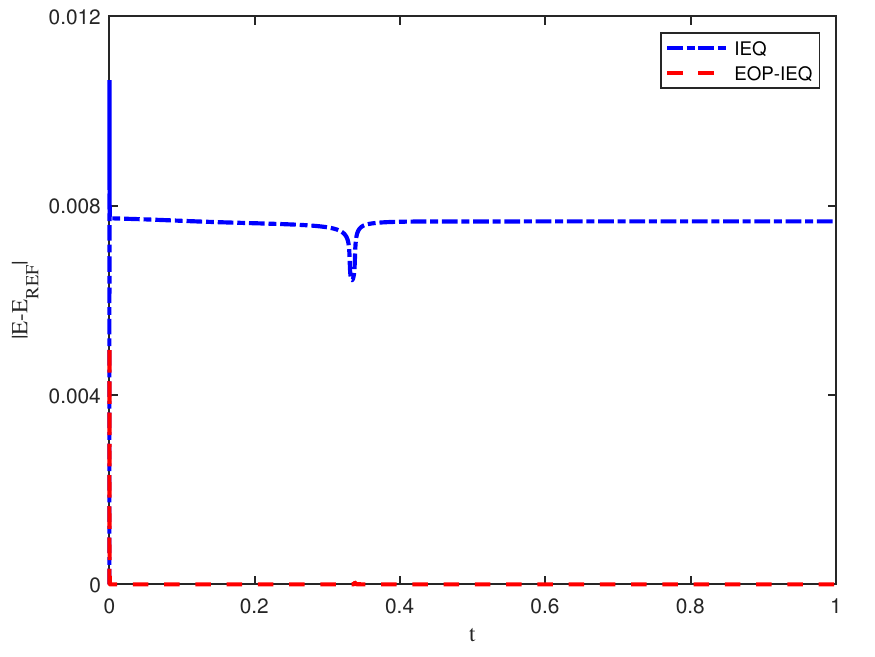}}
    \hfill
    \subfloat[]{\includegraphics[width=0.33\textwidth]{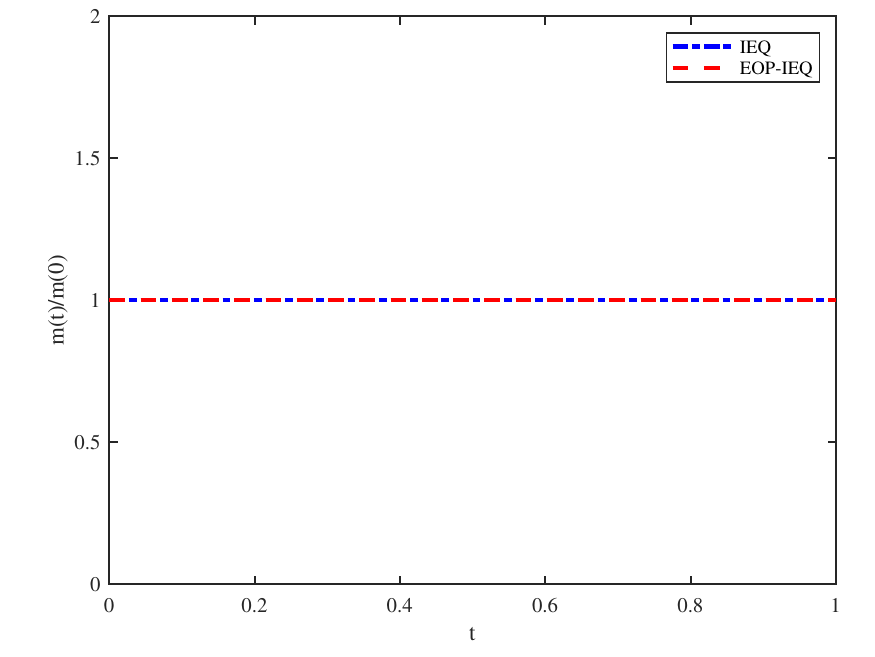}}
    \caption{Example 2B. 
    A comparison between the baseline IEQ/CN scheme and the EOP-IEQ/CN scheme for solving the CH equation. 
    %(a) Evolution of normalized energy under different numerical schemes;
    %(b) Evolution of relative mass deviation $m(t)/m(0)$ with various numerical schemes.
    (a) Normalized numerical energy comparisons of the two schemes;
    %(b) Evolution of the ratio of nonlinear free energy $\frac{1}{4}\lVert Q\left( \phi ^n \right) \rVert ^2$ and $\frac{1}{4}\lVert q^n \rVert ^2$;
    (b) Comparison of the difference between the numerical energy obtained by the two schemes and the reference energy; 
    %Comparison of numerical energy difference between the two methods;
    %Time evolution of free energy difference of the two methods;
    %Evolution of the ratio of nonlinear free energy $\frac{1}{4}\lVert Q\left( \phi ^n \right) \rVert ^2$ and $\frac{1}{4}\lVert q^n \rVert ^2$;
    (c) Time evolution of relative mass deviation $m(t)/m(0)$. %of the two methods.
    }
    \label{F26}
\end{figure}

\subsection{ Phase field crystal model }
Next, we investigate the phase field crystal (PFC) model \cite{ZHANG2013204,YANG20171116}, considering the free energy
\begin{align}
	\label{E-PFC}
	E\left( \phi \right) =\int_{\Omega}{\left( \frac{1}{4}\phi ^4+\frac{\beta-\varepsilon }{2}\phi ^2-\beta\left| \nabla \phi \right|^2+\frac{1}{2}\left( \Delta \phi \right) ^2 \right) \,\text{d}\boldsymbol{x}},
\end{align}
where $\beta$ the positive parameter and  $\beta>\varepsilon $.
Then the PFC equation can be expressed in the energy variational form \eqref{PF} of 
the free energy \eqref{E-PFC} with $\mathcal{G} = -M\Delta$ as follows
\begin{equation}
    \begin{aligned}
        {}&\phi _t=-M\Delta \mu, \\
        {}&\mu =\phi ^3+\left( \beta-\varepsilon \right) \phi +2\beta\Delta \phi +\Delta ^2\phi .
    \end{aligned}
\end{equation}
For the baseline IEQ method, we introduce the auxiliary variable
$q:=Q\left( \phi \right) =\phi^2$, then the general system \eqref{IEQ} is specified as
\begin{equation}
    \label{PFC-IEQ}
    \begin{aligned}
        {}&\frac{\partial \phi}{\partial t}=M\Delta \mu ,\\
        {}&\mu =q\phi +\left( \beta-\varepsilon \right) \phi +2\beta\Delta \phi +\Delta ^2\phi ,\\
        {}&\frac{\partial q}{\partial t}=2\phi \frac{\partial \phi}{\partial t}.
    \end{aligned}
\end{equation}
In addition, the PFC equation also satisfies the law of conservation of mass, i.e. \eqref{MM} is valid.

We simulate the complex dynamics of the growth of a polycrystal in a two-dimensional supercooled liquid
by using the EOP-IEQ method to further highlight its power for simulating long-time dynamics, 
involving the motion of the liquid crystal interface and the grain boundaries separating the crystals.
To define the initial configuration, we set up three microcrystals with different orientations as 
shown in Fig. \ref{F31} (a) in the computing domain $[0,800]^2$.
Other parameters are $(N_x,N_y)=(512,512)$, $\varepsilon =0.25 $, $\beta=1$,
$M=1$, $\tau=1e-2$ and $T=1500$. We use the results of the EOP-IEQ/CN scheme of $\tau=1e-4$ as the reference solution.
Fig. \ref{F31} shows microstructure evolution dynamics for triangular crystal phases driven by the PFC model
using the EOP-IEQ/CN scheme. 
The results also show that the different arrangements of crystals lead to defects and dislocations, 
similar to those reported in the literature \cite{Liu2020}.
The comparison of the numerical results of discrete modified energies is shown in Fig. \ref{F32} (a). 
The EOP-IEQ/CN scheme is more accurate than the baseline IEQ/CN scheme in Fig. \ref{F32} (b).
Most importantly, the ratio 
of nonlinear free energy $\lVert Q\left( \phi ^n \right) \rVert ^2/\lVert q^n \rVert ^2=1$
by using the energy-optimized technique in Fig. \ref{F32} (b), which shows that 
the EOP-IEQ method preserves the consistency of $Q( \phi ^n)$ and $q^n$.
Finally, Fig. \ref{F32} (c) indicates that the IEQ/CN scheme can ensure the conservation of mass regardless 
of the energy-optimized technique.

\begin{figure}[htp]
    \centering
    \subfloat[$t=0$]{\includegraphics[width=0.33\textwidth]{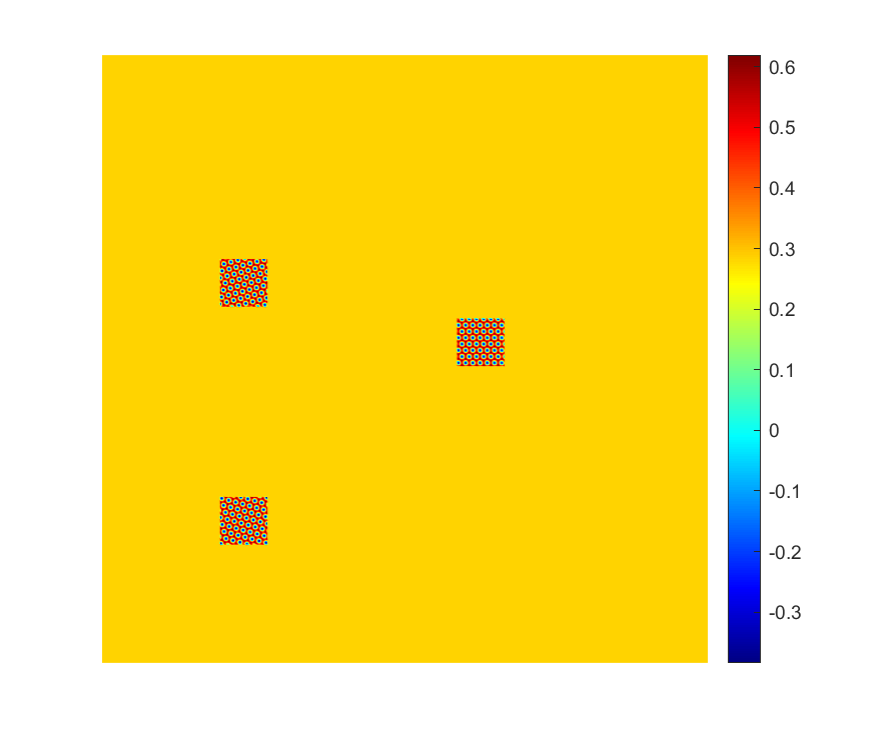}}
    \hfill
    \subfloat[$t=200$]{\includegraphics[width=0.33\textwidth]{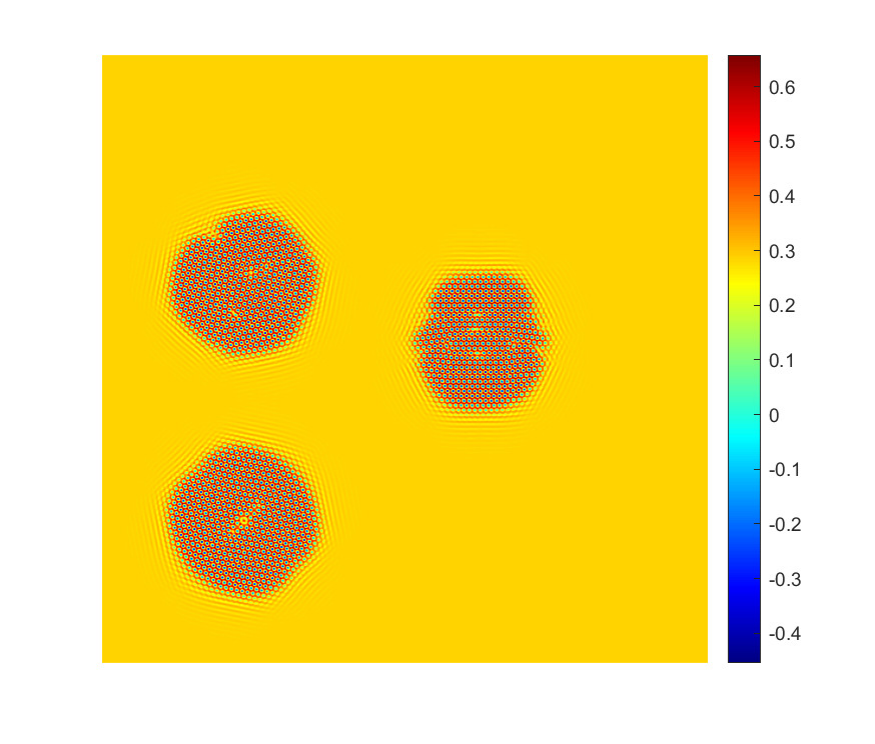}}
    \hfill
    \subfloat[$t=400$]{\includegraphics[width=0.33\textwidth]{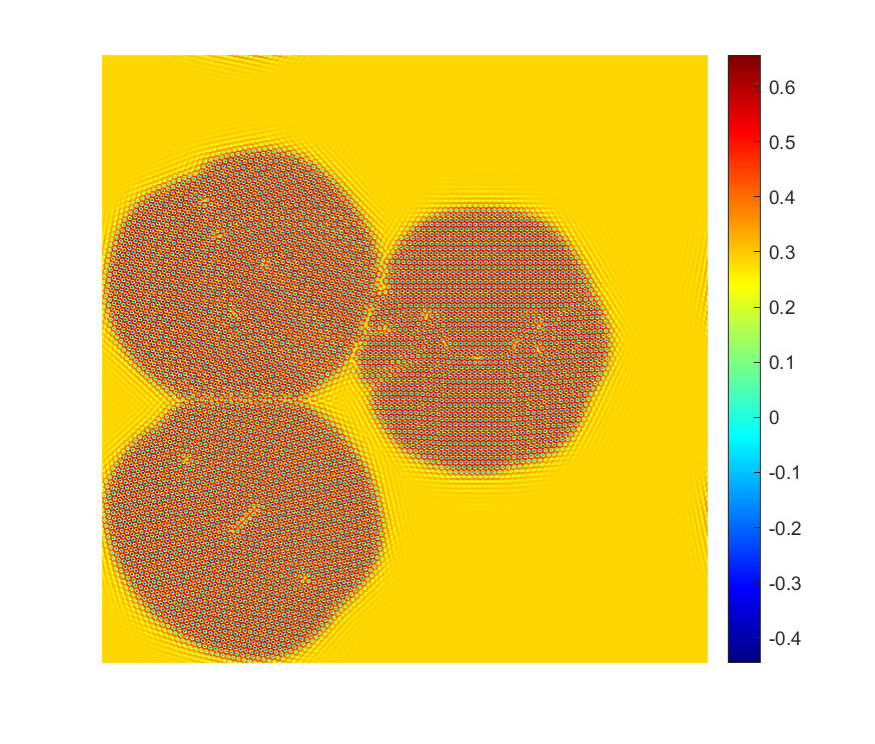}}
    \\
    \subfloat[$t=600$]{\includegraphics[width=0.33\textwidth]{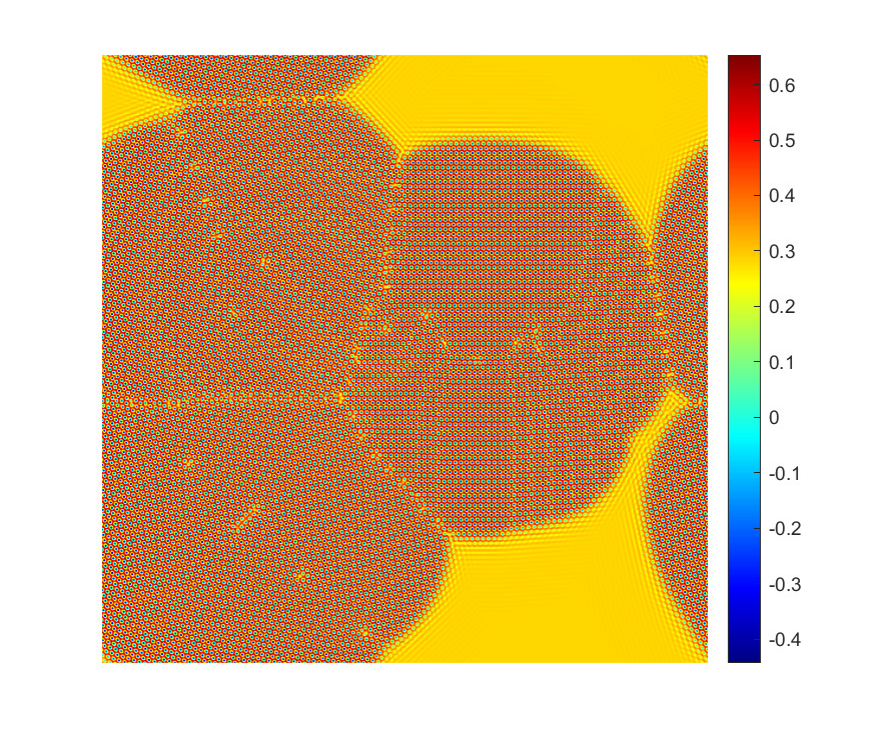}}
    \hfill
    \subfloat[$t=800$]{\includegraphics[width=0.33\textwidth]{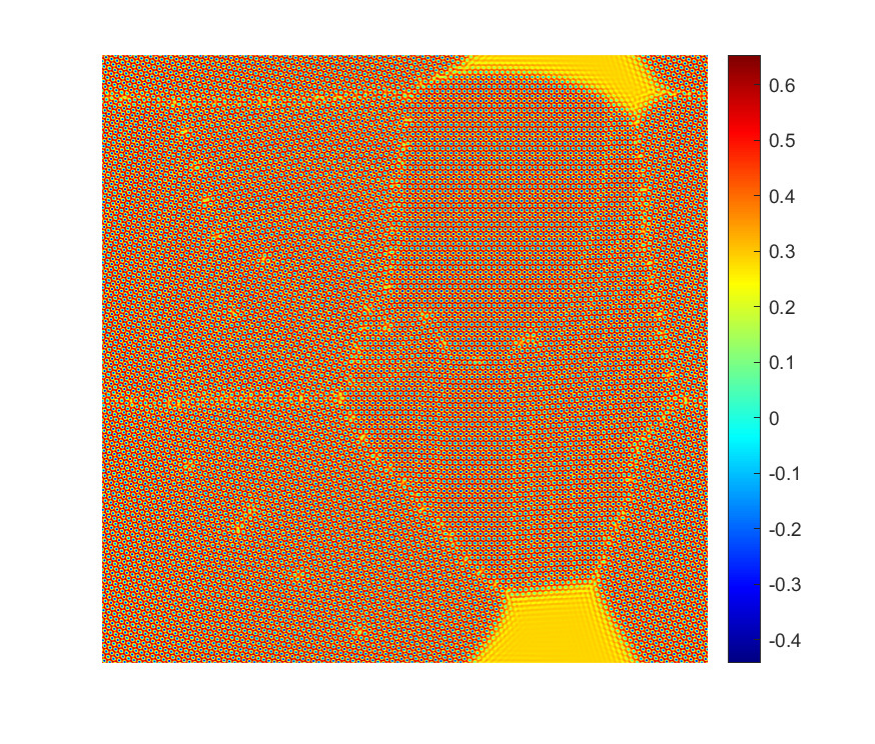}}
    \hfill
    \subfloat[$t=1500$]{\includegraphics[width=0.33\textwidth]{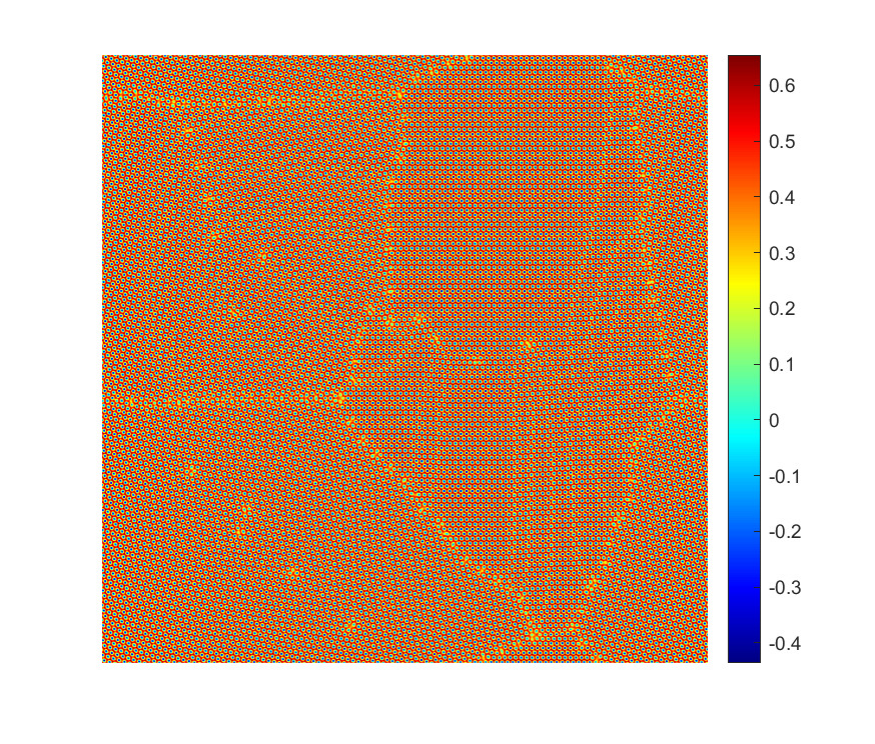}}
    \caption{Example 3. The dynamical behaviors of the crystal growth in a supercooled liquid.
    Profiles of the numerical solution  
    at $t = 0$, $200$, $400$, $600$, $800$, $1000$, $1500$, respectively.}
    \label{F31}
\end{figure}
\begin{figure}[htp]
    \centering
    \subfloat[]{\includegraphics[width=0.33\textwidth]{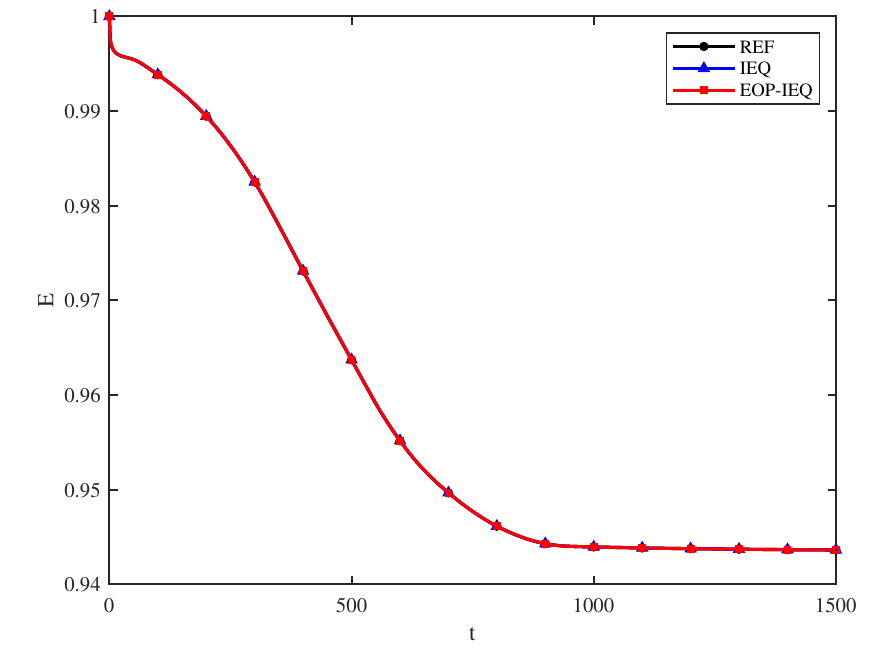}}
    \hfill
    \subfloat[]{\includegraphics[width=0.33\textwidth]{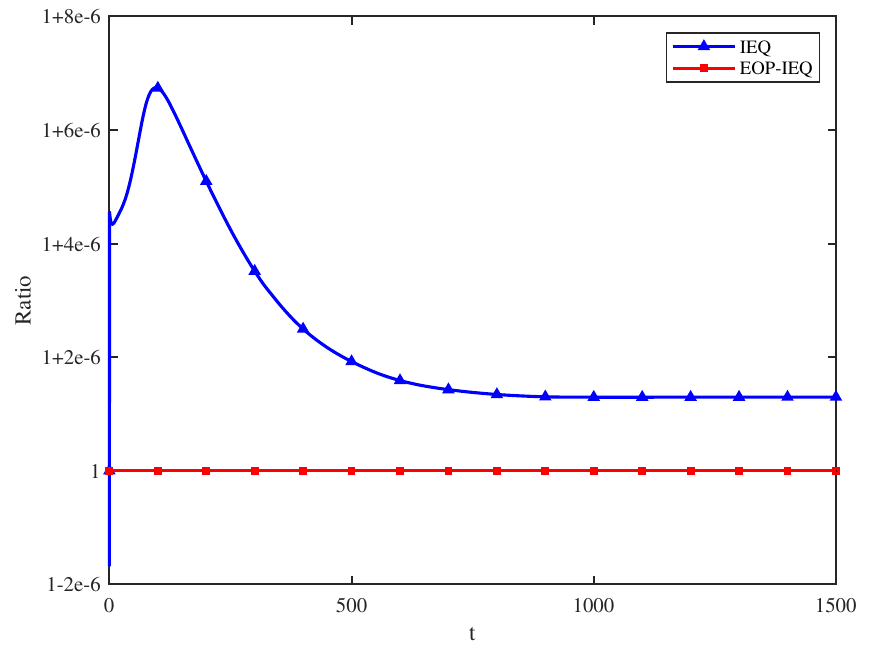}}
    \hfill
    \subfloat[]{\includegraphics[width=0.33\textwidth]{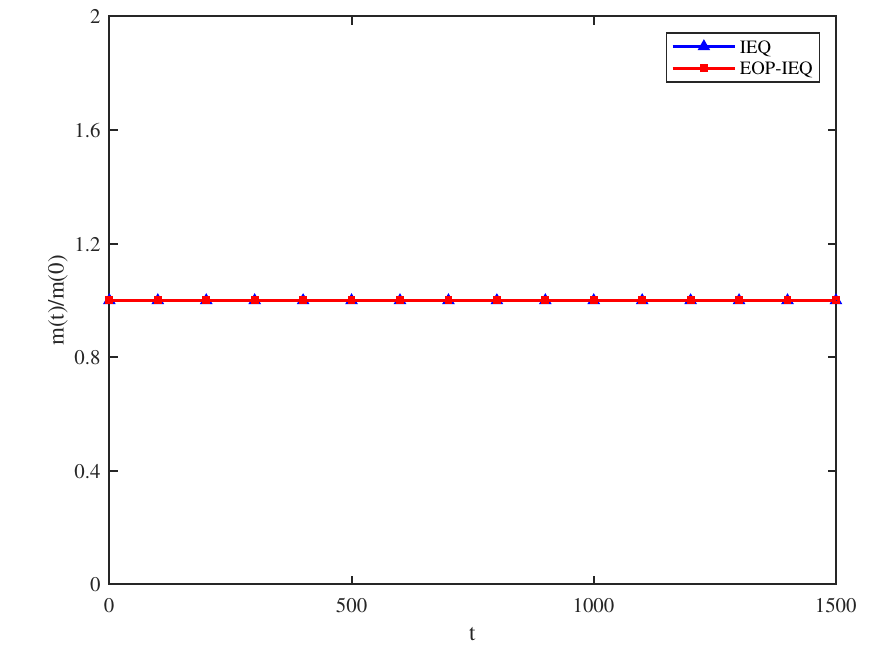}}
    \caption{Example 3. 
    A comparison between the baseline IEQ/CN scheme and EOP-IEQ/CN scheme for solving the PFC model. 
    (a) Normalized numerical energy comparisons of the two schemes;
    %(b) Evolution of the ratio of nonlinear free energy $\lVert Q\left( \phi ^n \right) \rVert ^2/4$ and $\lVert q^n \rVert ^2/4$;
    (b) Evolution of the ratio of nonlinear free energy $\lVert Q\left( \phi ^n \right) \rVert ^2/\lVert q^n \rVert ^2$;
    (c) Time evolution of relative mass deviation $m(t)/m(0)$.
    }
    \label{F32}
\end{figure}

\subsection{ Molecular beam epitaxy model }
Finally, we focus on the molecular beam epitaxy (MBE) model \cite{YANG2017104,Chen2015NewET,LI_LIU_2003} without slope section.
Considering the Ehrlich-Schwoebel energy 
\begin{align}
	\label{E-MBE}
    E\left( \phi \right) =\int_{\Omega}{\left( \frac{\varepsilon ^2}{2}\left( \Delta \phi \right) ^2-\frac{1}{2}\ln \left( 1+\left| \nabla \phi \right|^2 \right) \right) \,\text{d}\boldsymbol{x}}.
	%E\left( \phi \right) =\int_{\Omega}{\left( \frac{\varepsilon ^2}{2}\left( \Delta \phi \right) ^2-\frac{1}{2}F\left( \nabla \phi \right) \right) \,\text{d}\boldsymbol{x}},
    %\quad F\left( \nabla \phi \right) =\ln \left( 1+\left| \nabla \phi \right|^2 \right) .
\end{align}
Then the MBE equation can be expressed in the energy variational form \eqref{PF} of the free energy \eqref{E-MBE} with $\mathcal{G} = MI$,
which reads as
$$
\frac{\partial \phi}{\partial t}=-M\left( \varepsilon ^2\Delta ^2\phi +\nabla \cdot \left( \frac{1}{1+\left| \nabla \phi \right|^2} \nabla \phi \right) \right) .
$$
By introducing the auxiliary variable 
$q :=Q\left( \phi \right) 
=\sqrt{\ln \left( 1+\left| \nabla \phi \right|^2 \right) +C_0}$, 
%$q\left( \boldsymbol{x},t \right) :=Q\left( \phi\left( \boldsymbol{x},t \right) \right) 
%=\sqrt{\ln \left( 1+\left| \nabla \phi \right|^2 \right) +C_0}$, 
%=\sqrt{F(\nabla \phi\left( \boldsymbol{x},t \right)) +C_0}$, 
the MBE model can be equivalently written in quadratic form as follows
\begin{equation}
    \label{MBE-IEQ}
    \begin{aligned}
        {}&\frac{\partial \phi}{\partial t}=-M\left( \varepsilon ^2\Delta ^2\phi +\nabla \cdot \left( q\boldsymbol{H}(\phi) \right) \right) ,\\
        {}&\frac{\partial q}{\partial t}=\boldsymbol{H}(\phi)\cdot \nabla \frac{\partial \phi}{\partial t},
    \end{aligned}
\end{equation}
where
$$
\boldsymbol{H}(\phi)=\frac{\nabla \phi}{\left( 1+\left| \nabla \phi \right|^2 \right) Q(\phi) }.%\sqrt{\ln \left( 1+\left| \nabla \phi \right|^2 \right) +B}}
$$
Then the modified energy is
%$$
%\bar{E}\left( \phi ,q \right) =\int_{\Omega}{\left( \frac{\varepsilon ^2}{2}\left( \Delta \phi \right) ^2-\frac{1}{2}q^2+\frac{1}{2}C_0 \right) \,\text{d}\boldsymbol{x}}.
%$$
$$
\bar{E}\left( \phi ,q \right) =\frac{\varepsilon ^2}{2}\lVert \Delta \phi \rVert ^2-\frac{1}{2}\lVert q \rVert ^2+\frac{1}{2}C_0\left| \Omega \right|.
$$

%\begin{remark}
The energy-optimized technique for the MBE model without slope section \eqref{MBE-IEQ} is somewhat different from that of 
the general scheme \eqref{EOP/CN}, so we reconsider the EOP-IEQ method for the equivalent system \eqref{MBE-IEQ} based on the Crank-Nicolson formula.
%\end{remark}
\begin{corollary}[EOP-IEQ/CN scheme for the MBE model without slope section]
    \label{MBE/CN}
    For initial datas, we set $\phi ^0=\phi \left( \boldsymbol{x},0 \right)$,
    $q^0=\sqrt{\ln \left( 1+\left| \nabla \phi^0 \right|^2 \right) +C_0}$,
    %$q^0=\sqrt{F\left( \nabla \phi ^0 \right) +C_0}$,
    then we update $\phi^{n+1}$, $q^{n+1}$ via the following two steps:
    \begin{itemize}
        \item $\bf{Step~1}$ : To obtain the numerical solution $\phi^{n+1}$ by the baseline IEQ method:
        \begin{equation*}
            %\label{PFC-EOP-1}
            \begin{aligned}
                {}&\frac{\phi ^{n+1}-\phi ^n}{\tau}=-M\left( \varepsilon ^2\Delta ^2\frac{\phi ^{n+1}+\phi ^n}{2}+\nabla \cdot \left( \frac{\hat{q}^{n+1}+q^n}{2}\boldsymbol{\bar{H}}^{n+\frac{1}{2}} \right) \right) ,\\
                {}&\frac{\hat{q}^{n+1}-q^n}{\tau}=\boldsymbol{\bar{H}}^{n+\frac{1}{2}}\cdot \nabla \frac{\phi ^{n+1}-\phi ^n}{\tau},
            \end{aligned}
        \end{equation*}
        where  $
        \boldsymbol{\bar{H}}^{n+\frac{1}{2}}=\boldsymbol{H}\left( \frac{3}{2}\phi ^n-\frac{1}{2}\phi ^{n-1} \right) .
        $
        \item $\bf{Step~2}$ : %To optimize the intermediate solution $\hat{q}^{n+1}$ with an energy-optimized step,
        To update the auxiliary variable $q^{n+1}$ with an energy-optimized step:
         \begin{equation*}
            \begin{aligned}
                q^{n+1}=\lambda ^{n+1}Q(\phi^{n+1}),\quad 
                \lambda ^{n+1}=\left\{ \begin{array}{l}
                    1,E^{n+1}_2\ge 0,\\
                    \min \left\{ 1,\sqrt{\frac{-E^{n+1}_2}{E^{n+1}_1}} \right\} ,E^{n+1}_2<0,
                \end{array} \right. 
            \end{aligned}
        \end{equation*}
        where 
        \begin{align*}
            E^{n+1}_1=\lVert Q( \phi ^{n+1} ) \rVert ^2,\quad 
            E_{2}^{n+1}=\varepsilon ^2\lVert \Delta \phi ^n \rVert ^2-\lVert q^n \rVert ^2-\varepsilon ^2\lVert \Delta \phi ^{n+1} \rVert ^2.
        \end{align*}
    \end{itemize}
\end{corollary}
\begin{theorem}\label{Th-E-PFC}
    The scheme \ref{MBE/CN} is unconditionally energy stable in the sense that 
    \begin{align*}
        \frac{\varepsilon ^2}{2}\lVert \Delta \phi ^{n+1} \rVert ^2-\frac{1}{2}\lVert q^{n+1} \rVert ^2
        -\frac{\varepsilon ^2}{2}\lVert \Delta \phi ^n \rVert ^2+\frac{1}{2}\lVert q^n \rVert ^2
        =-\frac{1}{M\tau}\lVert \phi ^{n+1}-\phi ^n \rVert ^2\le 0.
    \end{align*}
\end{theorem}
The proof of the Theorem \ref{Th-E-PFC} is not particularly difficult but will be omitted.
%To get the deviation of the height function, 
We define the roughness measurement function $W(t)$ \cite{LI_LIU_2003} as follows
$$
W\left( t \right) =\frac{1}{\left| \Omega \right|^{\frac{1}{2}}}\lVert \phi \left( \boldsymbol{x},t \right) 
-\bar{\phi}\left( t \right) \rVert ,
\quad \bar{\phi}\left( t \right) =\frac{1}{|\Omega|}\int_{\Omega}{\phi \left( \boldsymbol{x},t \right)}\,\text{d}\boldsymbol{x}.
$$

In the last example, we also use the EOP-IEQ method and the baseline IEQ method to solve the MBE equation.
Numerical simulations are carried out in the rectangular region $[-\pi, \pi] ^2$.
We start from the classical benchmark problem
$$
\phi(x,y,0) = 0.1(\sin{3x}\sin{3y}+\cos{5x}\cos{5y}).
$$
The parameters are $(Nx,Ny) = (128,128)$, $M=0.5$, $\varepsilon = 0.1$, $C_0=1$, $\tau = 1e-2$ and $T=50$.
We use the results of the EOP-IEQ/CN scheme of $\tau=1e-4$ as the reference solution.
Several profiles of simulation process under the EOP-IEQ/CN scheme at $t=1, 10, 45$ in Fig. \ref{F41}.
We summarize the normalized numerical results of the modified energy evolution in 
Fig. \ref{F42} (a). It can be observed that the EOP-IEQ/CN scheme provides 
significantly more accurate results than the baseline IEQ/CN scheme in Fig. \ref{F42} (b).
Furthermore, Fig. \ref{F42} shows that the modified energy and roughness decay very quickly at the beginning, 
and after a long period of time, the roughness begins to increase. 
In general, the results are consistent with those in \cite{YANG2017104}.%the literature 
\begin{figure}[htp]
    \centering
    \subfloat[$t=1$]{\includegraphics[width=0.33\textwidth]{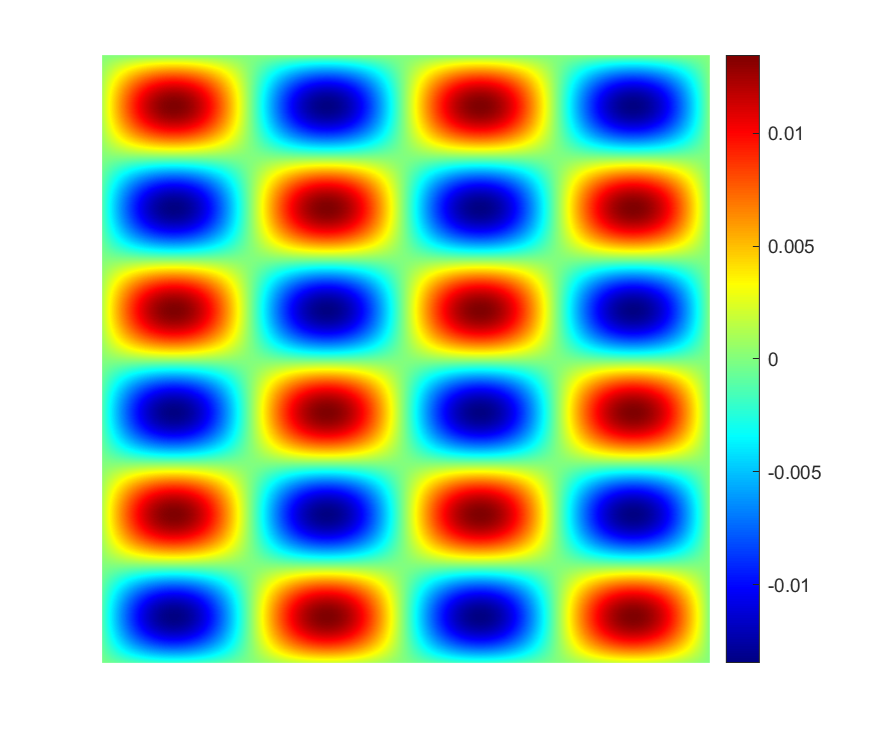}}
    \hfill
    \subfloat[$t=10$]{\includegraphics[width=0.33\textwidth]{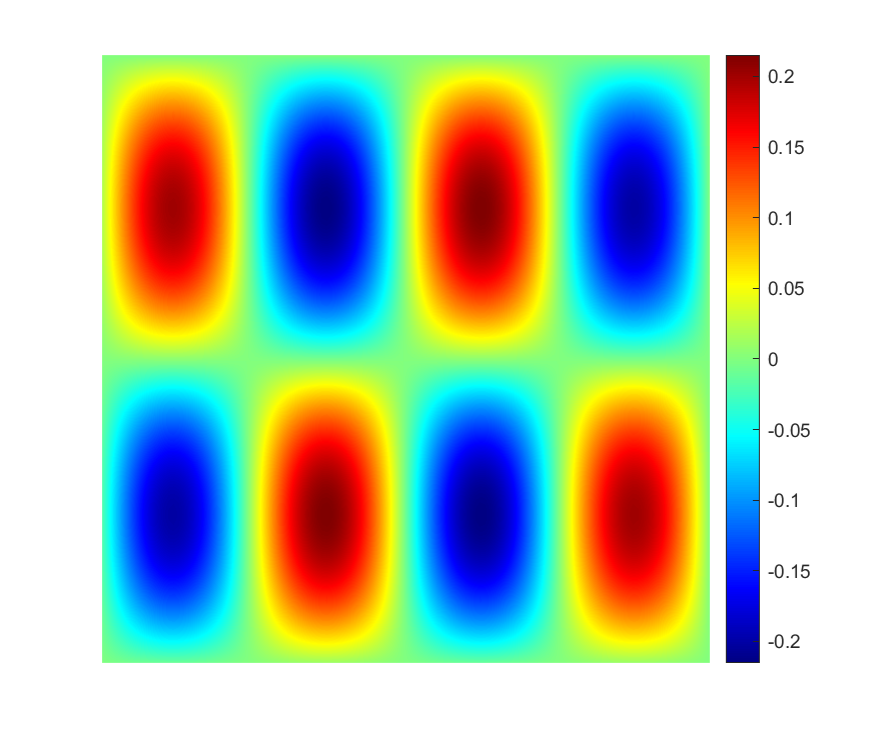}}
    \hfill
    \subfloat[$t=45$]{\includegraphics[width=0.33\textwidth]{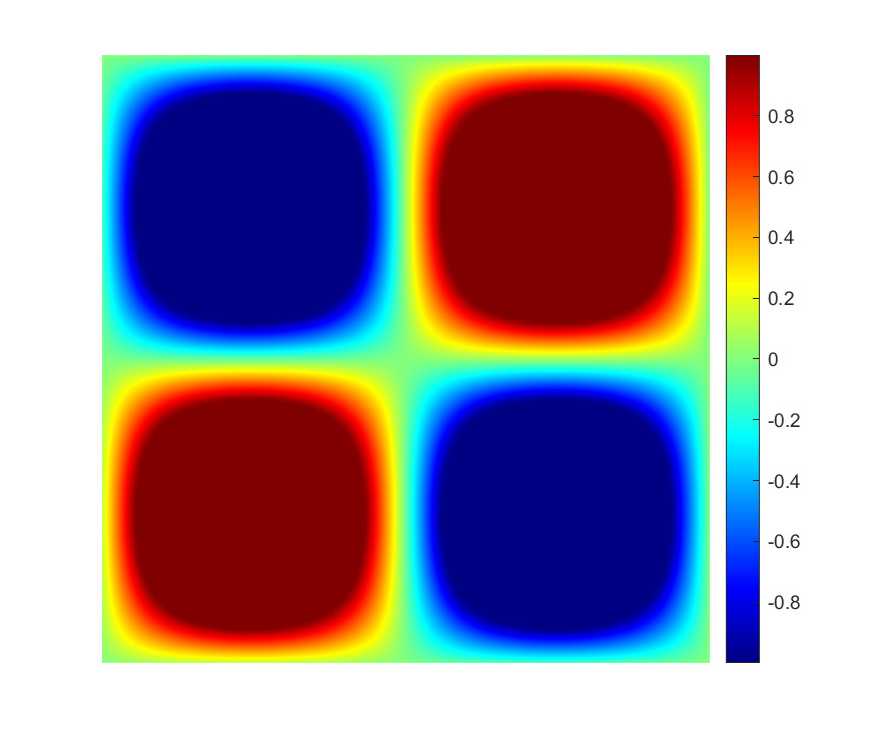}}
    \caption{Example 4. The isolines of numerical solutions of the height function $\phi$.
    Profiles of the numerical solution of $\phi$  
    at $t = 1$, $10$, $45$.}
    \label{F41}
\end{figure}

\begin{figure}[htp]
    \centering
    \subfloat[]{\includegraphics[width=0.33\textwidth]{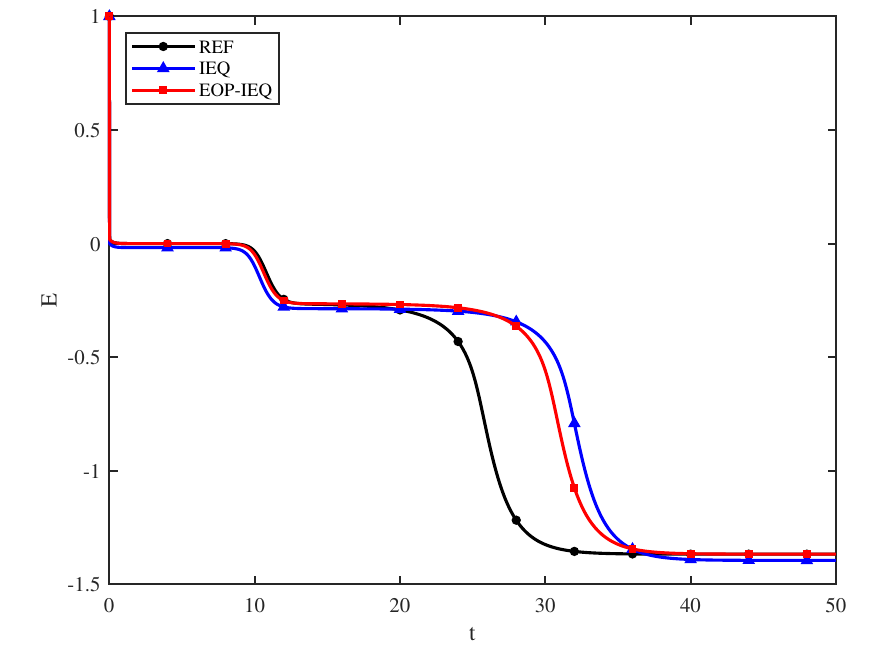}}
    \hfill
    \subfloat[]{\includegraphics[width=0.33\textwidth]{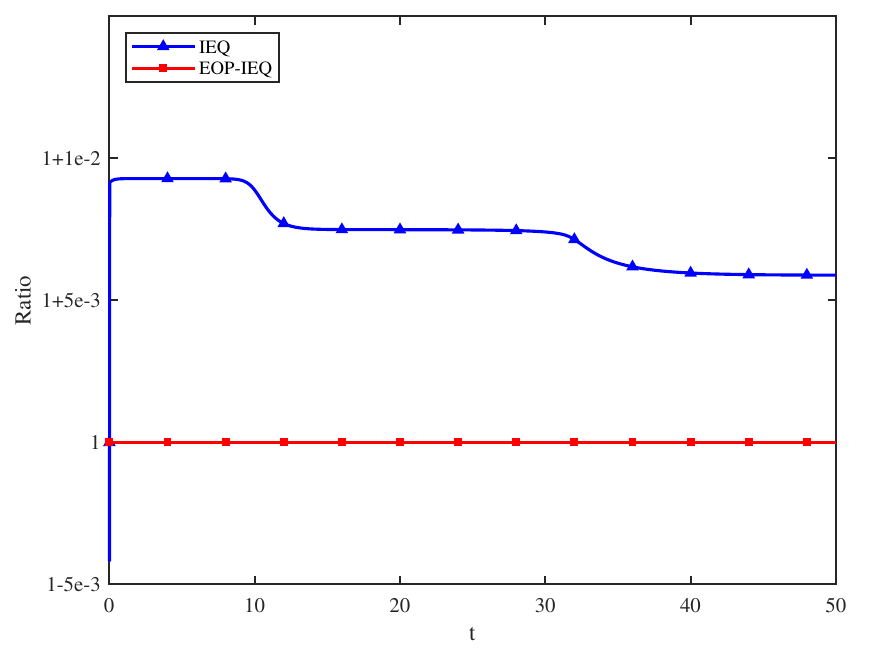}}
    \hfill
    \subfloat[]{\includegraphics[width=0.33\textwidth]{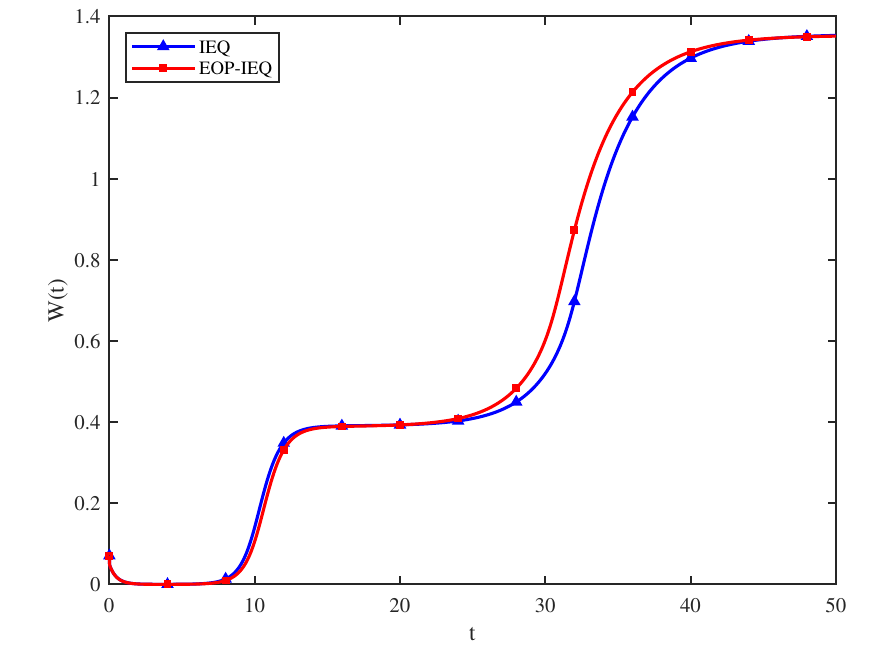}}
    \caption{Example 4. 
    A comparison between the baseline IEQ/CN scheme and the EOP-IEQ/CN scheme for solving the MBE model.
    (a) Normalized numerical energy comparisons of the two schemes;
    %(b) Evolution of the ratio of nonlinear free energy $\lVert Q\left( \phi ^n \right) \rVert ^2/4$ and $\lVert q^n \rVert ^2/4$;
    (b) Evolution of the ratio of nonlinear free energy $\lVert Q\left( \phi ^n \right) \rVert ^2/\lVert q^n \rVert ^2$;
    (c) Time evolution of the roughness $W(t)$.
    }
    \label{F42}
\end{figure}

\section{Conclusion}
In this paper, we propose an EOP-IEQ method from the perspective of original energy, 
and construct a new class of first- and second-order (in time) unconditionally energy stable schemes,
which inherit the advantages of the baseline IEQ and REQ schemes. 
The novel energy-optimized technique avoids the nonlinear optimization problems caused by the correction of auxiliary variables 
in REQ method, so our process of updating auxiliary variables is simpler and more efficient.
Ample numerical examples of gradient flow models, including AC equation, CH equation, PFC equation and MBE equation, 
verify the accuracy, efficiency and energy stability of the proposed schemes. 
Furthermore, our proposed numerical schemes are not only linearly solvable, 
but also have no significant deviation between the numerical solutions 
%calculated by the EOP-IEQ method 
and the reference solutions.
On the one hand, we observe that the modified energy 
using energy-optimized technique is significantly more closer to the original energy 
than the relaxation technique and baseline calculated step.
%through the comparison of IEQ, REQ and EOP-IEQ methods.
On the other hand, in most cases of long-term numerical simulation, 
the ratio of nonlinear free energy $q^n=Q( \phi ^n ), n\ge 0$ by using the EOP-IEQ method, 
indicating that the numerical solutions obtained by the EOP-IEQ method 
preserve the original energy dissipation law.
Overall, the main conclusion of this paper is that the energy-optimized technique 
is applicable to all available IEQ schemes in the literature
%with its easy-to-use and accuracy-improving properties.
with the natures of ease to use and improved accuracy.

\section*{CRediT authorship contribution statement}
Xiaoqing Meng: Methodology, Software, Validation, Formal analysis, Investigation, Writing-Original Draft, Visualization.
Aijie Cheng: Validation, Formal analysis, Resources, Writing-Review \& Editing, Supervision.
Zhengguang Liu: Conceptualization, Methodology, Validation, Formal analysis, Writing-Review \& Editing, Supervision.

\section*{Declaration of competing interest}
The authors declare that there are no conflicts of interest.
\section*{Data availability}
Data sharing not applicable to this article as no datasets were generated or analyzed during the current study.
\section*{Acknowledgment}
This work was supported by Natural Science Outstanding Youth Fund of Shandong Province (Grant No: ZR2023YQ007).
\bibliographystyle{elsarticle-num-names}
\bibliography{manuscript}
\end{document}